\documentclass[a4paper, reqno, 11pt]{amsart}
\makeatletter

\@addtoreset{equation}{section}
\makeatother
\usepackage{setspace}
\usepackage{amsmath}
\usepackage{amssymb}
\usepackage{latexsym}
\usepackage{amsthm}
\newtheorem{Thm}{Theorem}[section]

\newtheorem{Lem}[Thm]{Lemma}
\newtheorem{Prop}[Thm]{Proposition}

\theoremstyle{definition}

\newtheorem{Ass}[Thm]{Assumption}
\newtheorem{Def}[Thm]{Definition}
\newtheorem{Rem}[Thm]{Remark}
\newtheorem{Exa}[Thm]{Example}

\newcommand{\Fix}{\textup{fix}}

\newcommand{\R}{\mathbb{R}}

\begin{document}

\title[]{Some results for conjugate equations} 
\author{Kazuki Okamura}
\address{School of General Education, Shinshu University, 3-1-1, Asahi, Matsumoto, Nagano, JAPAN. TEL: +81-263-37-3066} 
\email{kazukio@shinshu-u.ac.jp} 

\maketitle

\begin{abstract}
In this paper we consider a class of conjugate equations, which generalizes de Rham's functional equations. 
We give sufficient conditions for existence and uniqueness of solutions under two different series of assumptions.   
We consider regularity of solutions. 
In our framework, two iterated function systems are associated with a series of conjugate equations. 
We state local regularity by using the invariant measures of the two iterated function systems with a common probability vector. 
We give several examples, especially an example such that infinitely many solutions exists, and a new class of fractal functions on the two-dimensional standard Sierpi\'nski gasket which are not harmonic functions or fractal interpolation functions. 
We also consider a certain kind of stability.   
\end{abstract}

\setcounter{tocdepth}{1}
\tableofcontents

\section{Introduction}

In this paper we consider the following functional equation. 
Let $X$ and $Y$ be non-empty sets. 
Let $I$ be a finite set. 
Assume that for $i \in I$, a subset $X_i \subset X$, and two maps $f_i : X_i \to X$ and $g_i : X \times Y \to Y$ are given.   
Now consider the solution $\varphi : X \to Y$ satisfying that 
\begin{equation}\label{conjugate} 
\varphi (f_i (x)) = g_i (x, \varphi(x)), \ x \in X_i, i \in I. 
\end{equation} 

The functional equation above is a generalization of de Rham's functional equation \cite{dR57} and  in the framework of iterative functional equations (cf. Kuczma-Choczewski-Ger \cite{KCG90}).  
\cite{KCG90} focuses on single equations, here we study not single but plural equations in a system.  
Here we consider solutions satisfying a series of plural equations simultaneously. 
This means that the set $I$ above contains at least two points. 
\cite{dR57} deals with the case that $X = [0,1]$ and $I = \{0,1\}$ and $g_i : Y \to Y$.  
De Rham's functional equation driven by affine functions and related functions such as Takagi functions have been considered in many papers. 
A few of related results are Hata \cite{Ha85}, Zdun \cite{Z01}, Girgensohn-Kairies-Zhang \cite{GKZ06}, Serpa-Buescu \cite{SB15a, SB15b, SB15c}, Shi-Yilei \cite{ST16},  Barany-Kiss-Kolossvary \cite{BKK18}, Allaart \cite{A}, etc. 
Here we do not give a detailed review of this topic. 
Recently, Serpa-Buescu \cite{SB17} considered \eqref{conjugate} and gave necessary conditions for existence of the solution of \eqref{conjugate}.
In the case that $X$ and $Y$ are metric spaces, they also gave sufficient conditions for existence and uniqueness. 
They also gave explicit formulae for the solution. 

This paper has three purposes. 
The first one is to consider sufficient conditions for existence and uniqueness of the solution of \eqref{conjugate}.
The second one is investigating regularity properties for the solution. 
The final one is considering a kind of stability of  the solution. 
 
This paper is organized as follows. 
In Section 2, we consider two different series of sufficient conditions for existence of the solution of \eqref{conjugate}, which are stated in Theorems \ref{exist-unique} and \ref{non-injective}. 
These results are similar to \cite[Theorem 1]{SB17}, however, in Theorem \ref{exist-unique}, we remove several assumptions of \cite[Theorem 1]{SB17}, and furthermore, in Theorem \ref{non-injective} we deal with the case that $f_i$ is {\it not} injective. 

In Section 3, we consider local regularity of solutions via invariant measures of iterated function systems. 
We focus on the case that for each $i \in I$, $X_i = X$ and the value of $g_i (x, y)$ does not depend on $x$ and furthermore $(X, \{f_i\}_{i \in I})$ and $(Y, \{g_i\}_{i \in I})$ are iterated function systems satisfying the open set conditions.
Informally speaking, we state in Theorem \ref{Thm-dim-Haus-2} that under certain conditions the solution of \eqref{conjugate} is fractal, or in another phrase, singular. 
The solution of \eqref{conjugate} measures how ``far" two iterated function systems $(X, \{f_i\}_{i \in I})$ and $(Y, \{g_i\}_{i \in I})$ are.
Although we do not need to introduce measures for the definition of \eqref{conjugate}, 
we state Theorem \ref{Thm-dim-Haus-2} by using integrals of certain functions with respect to the invariant measures of iterated function systems $(X, \{f_i\}_{i \in I})$ and $(Y, \{g_i\}_{i \in I})$ equipped with a common probability weight. 
We emphasize that Theorem \ref{Thm-dim-Haus-2} is applicable to the case that $f_i$ and $g_i$ are {\it non-affine} functions. 
Theorem \ref{Thm-dim-Haus-2} generalizes a modified statement of \cite[Theorem 1]{O16} and  is also related to \cite[Theorems 7.3 and 7.5]{Ha85} and \cite[Theorems 6 and 7]{Z01}. 
\cite{Ha85, Z01, O16} deal with the case that $X = [0,1]$, however, our result is also applicable to the case that $X$ is {\it not} $[0,1]$. 
We deal with the case that $X$ is the two-dimensional standard Sierpi\'nski gasket. 

In Sections 4 and 5, we give several examples.  
In Proposition \ref{interval-exist-con}, we consider the case that $f_i$ is {\it not} injective by applying Theorem \ref{non-injective}. 
In Example \ref{infinite-sol}, we give an example for the case that {\it infinitely many} solutions exist.
Example \ref{exa-overlap} deals with the case that the iterated function systems $(X, \{f_i\}_{i \in I})$ and $(Y, \{g_i\}_{i \in I})$ have overlaps. 
In Example \ref{exa-SG}, we give an example for the case that $X$ is the two-dimensional standard Sierpi\'nski gasket and $Y = [0,1]$. 
The solution is different from fractal interpolation functions on the Sierpi\'nski gasket considered by Celik-Kocak-Ozdemir \cite{CKO08}, Ruan \cite{R10} and Ri-Ruan  \cite{RR11}.  
The solution is not a harmonic function on the two-dimensional Sierpi\'nski gasket, and we would be able to call the solution a ``fractal function on a fractal". 

In Section 6, we consider a certain kind of stability of the solution. 
A little more specifically, we justify the following intuition: 
If two systems of \eqref{conjugate} are ``close" to each other, then, the two solutions of these systems are ``close" to each other.  
Since we do not put any algebraic structures on $X$ or $Y$, we cannot consider the Hyers-Ulam stability.  
We consider an alternative candidate of stability, by using the notion of {\it Gromov-Hausdorff convergence} on the class of metric spaces.
Finally in Section 7, we state four open problems concerning conjugate equations.  


\section{Existence and uniqueness}

Let $I$ be a finite set containing at least two distinct points. Let $I = \{0,1, \dots, N-1\}$. 
Let $X$ be non-empty sets and $(Y, d_Y)$ be a complete metric space. 
Assume that for each $i \in I$ a map $g_i : X_i \times Y \to Y$ is given.   

\begin{Ass}\label{ass-f}
For each $i \in I$, 
let $X_i \subset X$ and $f_i : X_i \to X$ be a map such that 
\[ X = \bigcup_{i \in I} f_i (X_i).\]  
\end{Ass}

Henceforth, if we do not refer to $X_i$, then, we always assume that $X_i = X$.   

Let $A$ be the set of contact points, that is, 
\[ A_i := \bigcup_{j \in I \setminus \{i\}} \bigcup_{x_j \in X_j} \{x_i \in X_i \mid f_i (x_i) = f_j (x_j) \} = \bigcup_{j \in I \setminus \{i\}} f_i^{-1}(f_j (X_j)), \] 
\[ A := \bigcup_{i \in I} A_i = \bigcup_{i \ne j} f_i^{-1}(f_j (X_j)). \]
Let 
\[ \widetilde A := A \cup \bigcup_{n \ge 1} \bigcup_{i_1, \dots, i_n \in I} f_{i_1} \circ \cdots \circ f_{i_n} (A). \]

\begin{Ass}\label{ass-new-compat}
Assume that there exists a {\it unique} bounded map $\varphi_0 : A \to Y$ 
such that \\
(i) \[ g_i (x_i, \varphi_0 (x_i)) = g_j (x_j, \varphi_0 (x_j)) \]  
holds for every $x_i \in A_i$ and $x_j \in A_j$ satisfying that $f_i (x_i) = f_j (x_j)$.\\
(ii)  If $f_{i_1} \circ \cdots \circ f_{i_n}(x) \in A$, then, 
\[ \varphi_0 (f_{i_1} \circ \cdots \circ f_{i_n}(x)) = g_{i_1}\left(f_{i_2} \circ \cdots \circ f_{i_n}(x), \cdot \right) \circ \cdots \circ g_{i_n}(x, \cdot) \circ \varphi_0 (x).  \]
\end{Ass}

We say that a function $\psi$ on $\R$ is increasing (resp. decreasing) if $\psi(t_1) \le \psi(t_2)$ (resp. $\psi(t_1) \ge \psi(t_2)$) whenever $t_1 \le t_2$, and, is strictly increasing (resp. strictly decreasing) if $\psi(t_1) < \psi(t_2)$ (resp. $\psi(t_1) > \psi(t_2)$) whenever $t_1 < t_2$. 

Let $(M, d)$ be a metric space. 
We say that $T : M \to M$ is {\it $\psi$-contractive} in the sense of Matkowski \cite{Ma75}
if $\psi : [0, +\infty) \to [0, +\infty)$ is an increasing function such that for any $t > 0$, 
\begin{equation}\label{n-iterate} 
\lim_{n \to \infty} \psi^n (t) = 0. 
\end{equation} 
and  
\[ d(Tx, Ty) \le \psi\left(d(x,y)\right), \ \textup{ for any $x, y \in M$}. \]
We say that $T$ is $\psi$-contractive in the sense of Browder (See Jachymski \cite{J97} for details) if we can take $\psi$ in the above as a {\it strictly} increasing function. 

Hereafter, if we say that a map is a weak contraction in the sense of Matkowski or Browder, 
then, we mean that for some $\psi$ the map is $\psi$-contractive in the sense of Matkowski or Browder, respectively.  
We remark that if \eqref{n-iterate} holds for an increasing function $\psi$, then, $\psi(t) < t$ for any $t > 0$.

\begin{Ass}\label{ass-g-wc}
For each $i \in I$ and $x \in X_i$, $g_i(x, \cdot)$ is $\phi_i$-contractive in the sense of Matkowski, where $\phi_i : [0, +\infty) \to [0, +\infty)$ is a map satisfying \eqref{n-iterate}.  
\end{Ass}

If $g_i(x, \cdot)$ does not depend on the choice of $x$, specifically, $g_i(x_1, y) = g_i(x_2, y)$ holds for any $x_1, x_2 \in X_i$ and $y \in Y$, 
then, we simply write $g_i (\cdot) = g_i (x, \cdot)$. 

\subsection{Result for the case that each $g_i$ depends on $x$} 

First we deal with the case that each $g_i$ depends on $x$. 

\begin{Ass}\label{ass-A} 
Assume that each $f_i$ is injective and
\begin{equation}\label{inv-A} 
\bigcup_{i \in I} f_i^{-1}(\widetilde A) \subset \widetilde A.  
\end{equation} 
\end{Ass}

\eqref{inv-A} is satisfied if $x \in A$ whenever $f_i (x) \in A$ for some $i \in I$, or $A = \emptyset$ , for example.

\begin{Thm}\label{exist-unique}
Under Assumptions \ref{ass-f}, \ref{ass-new-compat}, \ref{ass-g-wc} and \ref{ass-A},  
there exists a unique bounded map $\varphi : X \to Y$ 
such that $\varphi = \varphi_0$ on $A$, and \eqref{conjugate} holds for every $i \in I$ and $x \in X_i$. 
Here the boundedness of $\varphi$ means that $\textup{Image}(\varphi)$ is contained in a ball on $Y$, specifically,
\[ \sup_{x_1, x_2 \in X} d_Y (\varphi(x_1), \varphi(x_2)) < +\infty.  \] 
\end{Thm}

\begin{proof}
Let $\varphi_1$ be a map on $\widetilde A$ defined by $\varphi_1(y) := \varphi_0(y)$ if $y \in A$, and, 
\[ \varphi_1\left(f_{i_1} \circ \cdots \circ f_{i_n}(x)\right) := g_{i_1}\left(f_{i_2} \circ \cdots \circ f_{i_n}(x), \cdot \right) \circ \cdots \circ g_{i_n}(x, \cdot) \circ \varphi_0 (x), \ \ x \in A, \ n \ge 1. \]
This is well-defined due to Assumption \ref{ass-new-compat}. 
We now check this. 
Assume for $n, m \ge 1$ and $x, y \in A$, 
\[ f_{i_1} \circ \cdots \circ f_{i_n}(x) = f_{j_1} \circ \cdots \circ f_{j_m}(y). \]
then, 
$f_{i_2} \circ \cdots \circ f_{i_n}(x) \in A$ and $f_{j_2} \circ \cdots \circ f_{j_m}(x) \in A$. 
Now use Assumption \ref{ass-new-compat} (i). 
Assume for $n \ge 1, m = 0$ and $x, y \in A$,  
\[ f_{i_1} \circ \cdots \circ f_{i_n}(x) = y \in A. \]
then, use Assumption \ref{ass-new-compat} (ii). 

Recall Assumption \ref{ass-A}. 
If $x \in A_i \cap \widetilde A$, then, $f_i (x) \in \widetilde A$. 
By the definition of $\varphi_1$ and Assumption \ref{ass-new-compat}, 
\eqref{conjugate} holds for each $i \in I$ and $x \in \widetilde A$.

Hence, if $X = \widetilde A$, then we have \eqref{conjugate} for each $i \in I$ and $x \in X_i$.
   
Assume $X \setminus \widetilde A \ne \emptyset$.    
Let $\mathcal{B}$ be the set of bounded maps from $X \setminus \widetilde A$ to $Y$. 
By using Assumption \ref{ass-A} and the definition of $\widetilde A$, 
for each $x \in X \setminus \widetilde A$, 
there is a unique $i \in I$ such that $x \in f_i (X_i)$ and $f_i^{-1}(x) \in X \setminus \widetilde A$.  

Now we can define $T : \mathcal{B} \to \mathcal{B}$ such that 
\[ T[\varphi](x) := g_i \left(f_i^{-1}(x), \varphi(f_i^{-1}(x)) \right),  \ x \in X \setminus \widetilde A,  \]
for every $\varphi \in \mathcal{B}$.

We now put the following metric on $\mathcal{B}$. 
\[ D(\varphi_1, \varphi_2) := \sup_{x \in X \setminus \widetilde A} d_Y (\varphi_1(x), \varphi_2(x)) < +\infty, \ \ \varphi_1, \varphi_2 \in \mathcal{B}. \] 
Then, 
$(\mathcal{B}, D)$ is a complete metric space.  

\begin{Lem}
\eqref{n-iterate} holds for $\max_{i \in I} \phi_i$. 
\end{Lem} 

\begin{proof}[Proof of Lemma]
For each $t$ and $n$, there exists a sequence $(i_k)_{1 \le k \le n}$ such that 
\[ \left(\max_{i \in I} \phi_i\right)^n (t) = \phi_{i_1} \circ \cdots \circ \phi_{i_n} (t). \]
 
Then, by using the fact that $\phi_i (t) < t$ and the pigeonhole principle,  
\[ \phi_{i_1} \circ \cdots \circ \phi_{i_n} (t) \le \max_{i \in I} \left\{\phi_i^{\lfloor n/ N\rfloor} (t)\right\}. \]
Here $\lfloor n/ N\rfloor$ denotes the integer part of $n/ N$. 
Recall that \eqref{n-iterate} holds for $\phi_i$ for each $i \in I$. 
Thus \eqref{n-iterate} holds for $\max_i \phi_i$. 
\end{proof}

Return to the proof of Theorem \ref{exist-unique}.   
Now $T$ is $(\max_{i \in I} \phi_i)$-contractive. 
Hence, by the Matkowski fixed point theorem \cite{Ma75}, 
there exists a unique fixed point $\widetilde \varphi$ of $T$.
Now we have that for each $i \in I$, 
\begin{equation}\label{pre} 
\varphi = g_i \circ \varphi \circ f_i^{-1} \textup{ on $f_i (X_i)\setminus \widetilde A$}. 
\end{equation}    

By Assumption \ref{ass-A}, 
if $x \in X_i \setminus \widetilde A$, then, $f_i (x) \in X \setminus \widetilde A$. 
By this and \eqref{pre},
\eqref{conjugate} holds for $ \varphi = \widetilde \varphi$ on $X \setminus \widetilde A$. 

Let $\varphi$ be the map which equals $\varphi_1$ on $\widetilde A$ and $\widetilde \varphi$ on $X \setminus \widetilde A$. 
Therefore, \eqref{conjugate} holds. 
\end{proof} 

\begin{Rem}
(i) In \cite[Definition 1]{SB17}, it is stated that 
each value of the solution on $A$ has been previously determined by partially solving the system or by initial conditions. 
However, $X = A$ can happen, and in this case we may not be able to obtain values of the solution by the equation itself. 
We will give such examples below.\\
(ii) The assumption that $X$ is a bounded metric space in \cite[Theorem 1]{SB17} are removed. 
We do not put any topology on $X$, so in the above theorem we do not discuss the continuity of the solution.\\
(iii) Assumption \ref{ass-new-compat} is a necessary condition of Theorem \ref{exist-unique}.   
\end{Rem}

\subsection{Result for the case that $g_i$ does not depend on $x$}

In the proof of Theorem \ref{exist-unique}, 
the injectivity of $f_i$ is needed. 
Now we investigate the case that the injectivity of $f_i$ {\it fails}.

\begin{Ass}\label{ass-new-non-inj}
Assume $g_i = g_i (x, \cdot)$. 
Let $K$ be a unique compact subset of $Y$ such that $K = \cup_{i \in I} g_i (K)$. 
We assume that for any $x \in A$, 
there exists a {\it unique} infinite sequence $(i_n)_n \in I^{\mathbb{N}}$ such that 
\begin{equation}\label{x-in} 
x \in \bigcap_{n \ge 1} \textup{Image}(f_{i_1} \circ \cdots \circ f_{i_n}), 
\end{equation}  
and furthermore, 
\[ \bigcap_{n \ge 1} g_{i_1} \circ \cdots \circ g_{i_n}(K) =  \{\varphi_0 (x)\}. \]
\end{Ass}

\begin{Thm}\label{non-injective}
Under Assumptions \ref{ass-f}, \ref{ass-new-compat}, \ref{ass-g-wc}, and \ref{ass-new-non-inj},    
three exists a unique bounded solution $\varphi$ of \eqref{conjugate} such that $\varphi = \varphi_0$ on $A$. 
\end{Thm}

\begin{Rem}
If $g_i = g_i (x, \cdot)$, then, this is an extension of Theorem \ref{exist-unique}. 
In the following proof, we do not use any fixed point theorems. 
\end{Rem}

\begin{proof}[Proof of Theorem \ref{non-injective}]
First we show the existence. 
Let $x \in X$. 

We first remark that by Assumption \ref{ass-f}, 
for any $x \in X$ there exists at least one infinite sequence $(i_n)_n \in I^{\mathbb{N}}$ satisfying \eqref{x-in}. 

If there exists a {\it unique} infinite sequence $(i_n)_n \in I^{\mathbb{N}}$ satisfying \eqref{x-in}, 
then, by Assumption \ref{ass-g-wc}, 
we let $\varphi (x) \in Y$ be an element such that 
\begin{equation}\label{one-pt} 
\bigcap_{n \ge 1} g_{i_1} \circ \cdots \circ g_{i_n}(K) =  \{\varphi (x)\}. 
\end{equation} 

Otherwise, there exists a maximal integer $n = N(x)$ such that there exists a unique $(i_1, \dots, i_{n})$ such that $x \in \textup{Image}(f_{i_1} \circ \cdots \circ f_{i_n})$.  
Then there exist {\it at least two candidates of} $i_{N(x)+1} \in I$ and $x_{N(x)+1} \in A$ such that 
\begin{equation}\label{non-unique-x} 
x = f_{i_1} \circ \cdots \circ f_{i_{N(x) + 1}}(x_{N(x)+1}) 
\end{equation} 
and let 
\[ \varphi (x) := g_{i_1} \circ \cdots \circ g_{i_{N(x) + 1}}(\varphi_0(x_{N(x)+1})). \]
This is well-defined, that is, $\varphi (x)$ does not depend on the choice of $i_{N(x)+1} \in I$ and $x_{N(x)+1} \in A$ satisfying \eqref{non-unique-x}, 
due to Assumption \ref{ass-new-compat}.    

We need to show that $\varphi = \varphi_0$ on $A$.  
Let $x \in A$. 
By Assumption \ref{ass-new-non-inj}, 
there exists a unique infinite sequence $(i_n)_n \in I^{\mathbb{N}}$ satisfying \eqref{x-in}.   
Then \eqref{one-pt} holds.  
Now by Assumption \ref{ass-new-non-inj}, 
$\varphi (x) = \varphi_0 (x)$.  

Second we show the uniqueness.   
 Let $\varphi_1$ and $\varphi_2$ be bounded the solution of \eqref{conjugate}.  
Let $x \in X$. 

If there exists a unique infinite sequence $(i_n)_n \in I^{\mathbb{N}}$ satisfying \eqref{x-in}, 
then, by \eqref{conjugate} and the boundedness of $\varphi_1$ and $\varphi_2$,   
\[ \bigcap_{n \ge 1} g_{i_1} \circ \cdots \circ g_{i_n} \left(\textup{Image}(\varphi_1) \cup \textup{Image}(\varphi_2)\right) =  \{\varphi_1 (x)\} =  \{\varphi_2 (x)\}. \]
Hence $\varphi_1 (x) = \varphi_2 (x)$. 

Otherwise, there exists a maximal integer $n = N(x)$ such that there exists a unique $(i_1, \dots, i_{n})$ satisfying \eqref{x-in}.  
Then there exists at least two candidates of $i_{N(x)+1} \in I$ and $x_{N(x)+1} \in A$ satisfying \eqref{non-unique-x}.  
By this, \eqref{conjugate} and Assumption \ref{ass-new-compat}, 
\[ \varphi_i (x) = g_{i_1} \circ \cdots \circ g_{i_{N(x) + 1}}(\varphi_i (x_{N(x)+1})). \]
By the uniqueness for $\varphi_0$ in Assumption \ref{ass-new-compat}, 
\[ \varphi_0 = \varphi_1  = \varphi_2 \textup{ on } A.  \]  
Hence $ \varphi_1 (x)  =  \varphi_2 (x)$.  
\end{proof} 

\begin{Rem}
We are not sure whether there exist relationships between Assumptions \ref{ass-A} and \ref{ass-new-non-inj}. 
\end{Rem}


\section{Regularity}

In this section we always assume that $(X, d_X)$ and $(Y, d_Y)$ are two compact metric spaces 
and that there exist weak contractions $f_i, i \in I$, on $X$ and $g_i, i \in I$, on $Y$ in the sense of Browder such that 
\[ X = \bigcup_{i \in I} f_i (X) \textup{ and } Y = \bigcup_{i \in I} g_i (Y).\]  
Furthermore assume that there exists a unique solution $\varphi$ of \eqref{conjugate}. 
 
The aim of this section is to give sufficient conditions for each of the following:

\begin{Def}
Let $\alpha > 0$ and $a \in [0, +\infty]$. \\
(1) For a non-empty subset $U$ of $X$, 
we say that $(\alpha, U, a)$ holds if 
\[ \sup_{x_1, x_2 \in U, \ x_1 \ne x_2} \frac{d_Y (\varphi(x_1), \varphi(x_2))}{d_X (x_1, x_2)^{\alpha}} = a.  \]
(2) For $x \in X$, 
we say that $(\alpha, x, a)$ holds if 
\[ \limsup_{z \to x} \frac{d_Y (\varphi(z), \varphi(x))}{d_X (z, x)^{\alpha}} = a.  \]
\end{Def} 

For a subset $A$ of a metric space and for $s, \delta > 0$, let 
\[ \mathcal{H}^{\delta}_{s}(A) := \inf\left\{ \sum_{i = 1}^{\infty} \textup{diam}(U_i)^s \mid A \subset \bigcup_{i=1}^{\infty} U_i, \textup{diam}(U_i) \le \delta, \forall i   \right\}, \]
where we let $\textup{diam}(U)$ be the supremum of distances of two points of $U$.

This value is monotone decreasing with respect to $\delta$, we can let 
\[ \mathcal{H}_{s}(A) := \lim_{\delta \to 0+} \mathcal{H}^{\delta}_{s}(A).\] 
Then, we define the Hausdorff dimension of $A$ as follows: 
\[ \dim_H (A) := \sup\left\{s > 0 : \mathcal{H}_{s}(A) = +\infty\right\} = \inf\left\{s > 0 : \mathcal{H}_{s}(A) = 0\right\}.  \]
\begin{Lem}\label{Base-dim-Haus}
Let $(X, d_X)$ and $(Y, d_Y)$ be metric spaces. 
Let $A \subset X$ and $B \subset Y$ be non-empty. 
Let $\varphi : A \to B$ be a surjective map. 
If $\alpha > \dim_H (A) / \dim_H (B)$, then, 
$(\alpha, A, +\infty)$ holds. 
\end{Lem}

\begin{proof} 
Assume that $(\alpha, A, +\infty)$ fails.
Then, 
\[ \sup_{x_1, x_2 \in A, \ x_1 \ne x_2} \frac{d_Y (\varphi(x_1), \varphi(x_2))}{d_X (x_1, x_2)^{\alpha}} < +\infty, \]
that is, $\varphi$ is $\alpha$-H\"older continuous. 
Therefore, in the same manner as in the proof of \cite[Proposition 3.3]{Fal14}\footnote{It is stated for subsets of the Euclid space, but it  holds also for general metric spaces.},  
we can show that 
\[ \dim_H (\varphi(A)) \le \frac{\dim_H (A)}{\alpha}. \]
Since $\varphi$ is surjective, 
\[ \dim_H (B) \le \frac{\dim_H (A)}{\alpha}. \]
This contradicts the assumption that $\alpha > \dim_H (A) / \dim_H (B)$. 
\end{proof}

For a Borel probability measure $\mu$ on a metric space, 
let 
\[ \dim_H \mu := \inf\left\{\dim_H K \mid K : \textup{ Borel measurable } \mu(K) > 0 \right\}. \] 

Let $(X, d_X)$ and $(Y, d_Y)$ be two compact metric spaces. 
Assume that there exist weak contractions $f_i, i \in I$, on $X$ and $g_i, i \in I$, on $Y$ in the sense of Browder such that 
\[ X = \bigcup_{i \in I} f_i (X) \textup{ and } Y = \bigcup_{i \in I} g_i (Y).\]  

For $p_i, i \in I$, be numbers in $(0,1)$ such that $\sum_{i \in I} p_i = 1$, 
let $\mu_{\{p_i\}}$ and $\nu_{\{p_i\}}$ be two probability measures on $X$ and $Y$ 
such that 
\[ \mu_{\{p_i\}} = \sum_{i \in I} p_i \mu_{\{p_i\}} \circ f_i^{-1}, \textup{ and } \nu_{\{p_i\}} = \sum_{i \in I} p_i \nu_{\{p_i\}} \circ g_i^{-1}.\]    
The existences and uniquenesses of $\mu_{\{p_i\}}$ and $\nu_{\{p_i\}}$ under the assumption of the above theorem are assured by Fan \cite{Fan96}\footnote{For the case that each $f_i$ is a contraction, the existence and uniqueness of $\mu_{\{p_i\}_i}$ are classical and well-known (see Hutchinson \cite{Hu81}). The existence and uniqueness of $\mu_{\{p_i\}_i}$ for the case that each $f_i$ is a weak contraction is originally shown by \cite{Fan96} with use of an ergodic theorem. Alternative simpler proofs are given by \cite{AJS17, GMM, O18}}. 

\begin{Prop}\label{Thm-dim-Haus}
Assume that $\varphi$ is a solution of \eqref{conjugate}. 
Assume $\dim_H \nu_{\{p_i\}} > 0$. 
Let \[ \alpha > \frac{\dim_H \mu_{\{p_i\}}}{\dim_H \nu_{\{p_i\}}}.\]      
Then, $(\alpha, U, +\infty)$ holds for every non-empty open set $U$ of $X$. 
\end{Prop} 

This assertion is useful if we can know the values of $\dim_H \mu_{\{p_i\}}$ and $\dim_H \nu_{\{p_i\}}$. 
Fan-Lau \cite{FL99} might be useful under certain regularity assumptions for the two IFSs $(X, \{f_i\}_i)$ and $(Y, \{g_i\}_i)$. 

\begin{proof}
Let $\alpha > \dim_H \mu_{\{p_i\}} / \dim_H \nu_{\{p_i\}}$.  
Let $U$ be an arbitrarily open set of $X$. 
Take $\epsilon > 0$ such that 
\[ \alpha > \frac{\dim_H \mu_{\{p_i\}} + \epsilon}{\dim_H \nu_{\{p_i\}}}.\] 
Then, we can take $A \subset X$ such that 
\[ \dim_H A \le \dim_H \mu_{\{p_i\}} + \epsilon.\]   

Since $X$ is compact and $f_i$ are weak contractions, 
there exists  $i_1, \dots, i_n$ such that 
\[ f_{i_1} \circ \cdots \circ f_{i_n}(A) \subset U.\]    

Since each $f_i$ is Lipschitz continuous, 
\[ \dim_H f_{i_1} \circ \cdots \circ f_{i_n}(A) \le \dim_H A.\]

Take a Borel subset $B$ of $Y$ arbitrarily.
Then, by the definition of $\mu_{\{p_i\}} $
\[ \mu_{\{p_i\}} (\varphi^{-1} (B)) = \sum_{i \in I} p_i \mu_{\{p_i\}}  \left( f_i^{-1} (\varphi^{-1} (B)) \right). \]
By \eqref{conjugate}, we have that for each $i \in I$, 
\[ f_i^{-1} (\varphi^{-1} (B)) = (\varphi \circ f_i )^{-1} (B) =  (g_i \circ \varphi)^{-1} (B) = \varphi^{-1} (g_i^{-1}(B)). \]
Hence, 
\[ \mu_{\{p_i\}} (\varphi^{-1} (B)) = \sum_{i \in I} p_i \mu_{\{p_i\}} \varphi^{-1} (g_i^{-1}(B)).  \]
By this and the uniqueness of self-similar measures established in \cite{Fan96},   
we have that 
\[ \nu_{\{p_i\}} = \mu_{\{p_i\}} \circ \varphi^{-1}, \] 
and hence,   
\[ \nu_{\{p_i\}}\left(\varphi (f_{i_1} \circ \cdots \circ f_{i_n}(A)) \right) \ge \mu_{\{p_i\}}(f_{i_1} \circ \cdots \circ f_{i_n}(A)). \]

By the definition of $\mu$, 
\[ \mu_{\{p_i\}}(f_{i_1} \circ \cdots \circ f_{i_n}(A)) = \sum_{j_1, \dots, j_n \in I} p_{j_1} \cdots p_{j_n} \mu_{\{p_i\}}\left((f_{j_1} \circ \cdots \circ f_{j_n})^{-1}(f_{i_1} \circ \cdots \circ f_{i_n}(A)) \right) \]
\[ \ge p_{i_1} \cdots p_{i_n} \mu_{\{p_i\}}(A) > 0. \]

By the definition of $\dim_H \nu_{\{p_i\}}$, 
\[ \dim_H \varphi (f_{i_1} \circ \cdots \circ f_{i_n}(A)) \ge \dim_H \nu_{\{p_i\}}. \]

Therefore, 
\[ \alpha > \frac{\dim_H \mu_{\{p_i\}} + \epsilon}{\dim_H \nu_{\{p_i\}}} \ge \frac{\dim_H f_{i_1} \circ \cdots \circ f_{i_n}(A)}{\dim_H \varphi (f_{i_1} \circ \cdots \circ f_{i_n}(A))}. \]

Now the assertion follows from  Lemma \ref{Base-dim-Haus}.    
\end{proof}

\subsection{Local regularity}

Let $I = \{0,1,\dots, N-1\}$. 

\subsubsection{Assumptions for $(X, \{f_i\}_i)$}

Hereafter, we denote the set of linear bounded transformations on a separable Banach space $E$ by $B(E, E)$. 

\begin{Ass}\label{ass-x-f}
(i) $E_1$ is a separable Banach space.\\ 
(ii) $X$ is a compact subset of $E_1$ such that 
\[ \bigcup_{i \in I} f_i (X) = X  \] 
and its interior is non-empty.\\
(iii) each $f_i$ is weakly contractive on $X$.\\    
(iv) There exists the total derivative of $f_i$ at $x \in X$, which is denoted by $Df_i(x) \in B(E_1, E_1)$.\\ 
(v) Assume that for each $i$ and $x \in X$,   
$Df_i (x)$ is non-degenerate, specifically, 
\[ \inf_{z \ne 0} \frac{|Df_i (x)(z)|}{|z|} > 0, \]
where $| \cdot |$ is the norm of $E_1$. 
If so, $Df_i (x)$ is invertible and 
\[ \| (Df_i (x))^{-1} \|^{-1} = \left(\sup_{w \in Df_i (x)(E_1), w \ne 0} \frac{|(Df_i (x))^{-1}w|}{|w|} \right)^{-1} \]
\[= \inf_{w \in Df_i (x)(E_1), w \ne 0} \frac{|w|}{|(Df_i (x))^{-1}w|} = \inf_{z \ne 0} \frac{|Df_i (x)(z)|}{|z|}. \]
\end{Ass}  

Hereafter, $\lfloor z \rfloor$ denotes the maximal integer which does not larger than a real number $z$. 

\begin{Ass}[regularity property]\label{ass-reg}
We say that $(X, \{f_i\}_i)$ satisfies a regularity property if for any $\epsilon > 0$
there exists $C > 0$ such that for every $n \ge 0$, $i_1, \dots, i_{\lfloor n(1+\epsilon)\rfloor }$, 
\[ \textup{dist} \left(f_{i_1} \circ \cdots \circ f_{i_{\lfloor n(1+\epsilon)\rfloor }}(X),  X \setminus f_{i_1} \circ \cdots \circ f_{i_{n}}(X) \right) 
\ge C \textup{diam}\left(f_{i_1} \circ \cdots \circ f_{i_{\lfloor n(1+\epsilon)\rfloor }}(X)\right). \]
\end{Ass} 

\begin{Ass}[Existence of two distant points]\label{ass-diam} 
There exists $c > 0$ such that for every $n \ge 0$ and $(i_1, \dots, i_n)$, 
\[ d_X \left(f_{i_1} \circ \cdots \circ f_{i_{n}}(\textup{fix}(f_0)), f_{i_1} \circ \cdots \circ f_{i_{n}}(\textup{fix}(f_{N-1}))\right) \ge c \ \textup{diam}\left(f_{i_1} \circ \cdots \circ f_{i_{n}}(X)\right). \]
\end{Ass} 

\begin{Ass}[Measure separation property]\label{ass-msp} 
Assume that 
\[ \mu_{\{p_i\}}\left( f_{i_1} \circ \cdots \circ f_{i_n}(X) \cap f_{j_1} \circ \cdots \circ f_{j_n}(X) \right) = 0 \]
holds whenever $(i_1, \dots, i_n) \ne (j_1, \dots, j_n)$. 
\end{Ass}

\subsubsection{Assumptions for $(Y, \{g_i\}_i)$}

The following corresponds to Assumption \ref{ass-x-f}. 

\begin{Ass}\label{ass-y-g}
(i) $E_2$ is a separable Banach space.\\ 
(ii) $Y$ is a closed subset of $E_2$ such that 
\[ \bigcup_{i \in I} g_i (Y) = Y \] 
and its interior is non-empty.\\
(iii) each $g_i$ is weakly contractive on $Y$.\\    
(iv) There exists the total derivative of $g_i$ at $y \in Y$, which is denoted by $Dg_i(y) \in B(E_2, E_2)$.\\ 
(v) Assume that for each $i$ and $y \in Y$,   
$Dg_i (y)$ is non-degenerate, specifically, 
\[ \inf_{z \in E_2 \setminus \{0\}} \frac{|Dg_i (y)(z)|}{|z|} > 0. \]
If so, $Dg_i (y)$ is invertible and 
\[ \| (Dg_i (y))^{-1} \|^{-1} = \inf_{z  \in E_2 \setminus \{0\}} \frac{|Dg_i (y)(z)|}{|z|}. \]
(vi) $\textup{fix}(g_0) \ne \textup{fix}(g_{N-1})$.  
\end{Ass}

$Y$ is not necessarily compact. 

\subsubsection{Local regularity for solution}

The following gives a local regularity for $\varphi$ at $\mu_{\{p_i\}}$-almost every each point. 

\begin{Thm}\label{Thm-dim-Haus-2}
Let $\varphi : X \to Y$ be a continuous solution of \eqref{conjugate}.  
Under Assumptions \ref{ass-x-f} - \ref{ass-y-g}, 
we have the following:\\
(i) If \[ \alpha < \frac{\sum_{i \in I} p_i  \int_{Y} \log (1/\| Dg_i(y)\|) \nu_{\{p_i\}}(dy)}{\sum_{i \in I} p_i  \int_{X} \log \|Df_i(x)^{-1}\| \mu_{\{p_i\}}(dx)},  \]  
then $(\alpha, x, 0)$ holds for $\mu_{\{p_i\}}$-a.e. $x$.   \\
(ii) If \[ \beta > \frac{\sum_{i \in I} p_i  \int_{Y} \log \| Dg_i(y)^{-1} \| \nu_{\{p_i\}}(dy)}{\sum_{i \in I} p_i  \int_{X} \log (1/\|Df_i(x)\|) \mu_{\{p_i\}}(dx)}, \] 
then, $(\beta, x, +\infty)$ holds for $\mu_{\{p_i\}}$-a.e. $x$.     
\end{Thm} 

\begin{Rem}
(i) Assumptions \ref{ass-x-f}-\ref{ass-y-g} may not assure that there exists a solution of \eqref{conjugate}. \\
(ii) In the statement of \cite[Theorem 1]{O16}, we had to assume that $Dg_i (y)$ is non-degenerate as in Assumption \ref{ass-y-g} (v), 
but actually we did not. (Notation here is different from \cite{O16}.) \\
(iii) However, \cite{O16} considers the case $X = [0,1]$ only. 
Here we give a generalization for more general self-similar sets containing ($d \ge 2$-dimensional) standard Sierpi\'nski gaskets and carpets, for example.  
Assumptions \ref{ass-x-f}-\ref{ass-y-g} hold for $d$-dimensional Sierpi\'nski gaskets and carpets. \\
(iv) \cite[Theorems 7.3 and 7.5]{Ha85} correspond to \cite[Theorems 6 and 7]{Z01}, respectively. 
They are essentially same, but \cite[Theorems 7.3 and 7.5]{Ha85} is a little more general than \cite[Theorems 6 and 7]{Z01}. 
\end{Rem}

\begin{proof}[Proof of Theorem \ref{Thm-dim-Haus-2}] 
For $i \in I$, let 
\[ \widetilde X(i) := f_i (X) \setminus \bigcup_{n \ge 1} \bigcup_{(i_1, \dots, i_n) \ne (j_1, \dots, j_n)} f_{i_1} \circ \cdots \circ f_{i_n}(X) \cap f_{j_1} \circ \cdots \circ f_{j_n}(X). \]
Assume that for some $i$, $x \in \widetilde X(i)$. 
Let $T(x) := f_i^{-1}(x)$ and $I_1 (x) := i$.  
Then, for some $i_2$, $T(x) \in \widetilde X(i_2)$.   
Hence we can define $T(T(x)) = f_{i_2}^{-1}(T(x))$ and $I_2(x) := I_1 (T(x)) = i_2$.  
By repeating this, we have an infinite sequence $(I_n(x))_{n \ge 1} \in I^{\mathbb{N}}$ for $x \in \bigcup_{i \in I} \widetilde X(i)$.    
By Assumption \ref{ass-msp}, $\mu_{\{p_i\}_i}\left( \bigcup_{i \in I} \widetilde X(i) \right) = 1$. 

Let $\theta$ be the one-sided shift on $I^{\mathbb{N}}$.  
By \cite[Proposition 1.3]{FL99}, 
there exists a measurable map $\pi : I^{\mathbb{N}} \to X$ such that 
\[ T \circ \pi = \pi \circ \theta, \textup{ on } \pi^{-1}\left( \bigcup_{i \in I}  \widetilde X(i) \right). \]
Here we put the cylindrical $\sigma$-algebra on $I^{\mathbb{N}}$.

For $i \in I$, let $\sigma_i : I^{\mathbb{N}}  \to I^{\mathbb{N}}$ such that $\sigma_i (\omega) = i \omega$, which is a concatenation of $i$ and $\omega$. 
Let $\eta_{\{p_i\}}$ be a specific probability measure on $I^{\mathbb{N}}$ such that 
\[ \eta_{\{p_i\}} = \sum_i p_i \eta_{\{p_i\}} \circ \sigma_i^{-1}.\]  
Then $\theta$ is invariant and ergodic with respect to  $\eta_{\{p_i\}}$. 

By \cite[Proposition 1.3 (ii)]{FL99}, 
\[ \mu_{\{p_i\}} = \eta_{\{p_i\}} \circ \pi^{-1}.\]    
Thus we have that $\mu_{\{p_i\}}$ is invariant and ergodic with respect to $T$, 
and furthermore $\{I_i\}_i$ are i.i.d. under $\mu_{\{p_i\}}$.    

We first show assertion  (ii).  
For $x \in \bigcup_{i \in I} \widetilde X(i)$, let 
\[ F_n (x) :=  \left|f_{I_1(x)} \circ \cdots \circ f_{I_n(x)}(\textup{fix}(f_0)) - f_{I_1(x)} \circ \cdots \circ f_{I_n(x)}(\textup{fix}(f_{N-1}))\right|, \]
and, 
\[ G_n (x) := \left| g_{I_1(x)} \circ \cdots \circ g_{I_n(x)}(\textup{fix}(g_0)) - g_{I_1(x)} \circ \cdots \circ g_{I_n(x)}(\textup{fix}(g_{N-1})) \right|. \]

If $n = 0$, let $G_0 (x) := 1$. 
Since $I_k(x) = I_1 (T^{k-1}(x))$, by \cite[Lemma 3.1]{O16},  
\begin{equation}\label{tx-f-estimate}
\frac{1}{\| (Df_{I_1(x)}(Tx))^{-1}\|} + o_n (1) \le \frac{F_n (x)}{F_{n-1}(T(x))} \le \| Df_{I_1(x)}(Tx)\| + o_n (1),  
\end{equation}
and 
\begin{equation}\label{tx-g-estimate} 
\frac{1}{\| (Dg_{I_1(x)}(\varphi(Tx)))^{-1}\|} + o_n (1) \le \frac{G_n (x)}{G_{n-1}(T(x))} \le \| Dg_{I_1(x)}(\varphi(Tx))\| + o_n (1),  
\end{equation}
where the small order $o_n (1)$ is uniform with respect to $x$. 

For every $\epsilon > 0$,
there exists $N$ such that for every $n > N$, 
\[ \log G_n (x) = \sum_{i=1}^{n} \log \frac{G_i (T^{n-i}(x))}{G_{i-1}(T^{n-i +1}(x))} \]
\[ \ge \sum_{i=1}^{n-N} \log \left(1/\|Dg_{I_1(T^{n-i}(x))}(\varphi(T^{n-i +1}(x)))^{-1}\|\right) - \log(1+\epsilon), \]
and, 
\[ log F_n (x) = CN +  \sum_{i=1}^{n-N} \log \frac{F_i (T^{n-i}(x))}{F_{i-1}(T^{n-i +1}(x))} \]
\[ \le CN +  \sum_{i=1}^{n-N} \| Df_{I_1(T^{n-i}(x))}(T^{n-i + 1}(x))\| + \log(1+\epsilon), \]
where $C$ is a constant independent from $\epsilon$. 

Now by using these inequalities and the Birkhoff ergodic theorem
it holds that  $\mu_{\{p_i\}}$-a.e.$x$, 
\[ \liminf_{n \to \infty} \frac{\log G_n (x)}{n} \ge \int_{X} \log \left(1/\|Dg_{I_1(x)}(\varphi(Tx))^{-1}\|\right)  \mu_{\{p_i\}}(dx) \] 
\[ = \sum_i p_i \int_{X} \log \left(1/\|Dg_{i}(\varphi(Tx))^{-1}\|\right)  \mu_{\{p_i\}}(dx) \]
\[ = \sum_i p_i \int_{X} \log \left(1/\|Dg_{i}(\varphi(x))^{-1}\|\right)  \mu_{\{p_i\}}(dx) \]
\[ = \sum_i p_i \int_{Y} \log 1/\|Dg_{i}(y)^{-1}\| \nu_{\{p_i\}}(dy), \] 
and, 
\[ \limsup_{n \to \infty} \frac{\log F_n (x)}{n} \le \int_{X} \log \|Df_{I_1(x)}(Tx)\|  \mu_{\{p_i\}}(dx) \]
\[ = \sum_i p_i \int_{X} \log \|Df_{i}(x)\| \mu_{\{p_i\}}(dx) \]
\[ \le \frac{1}{\beta} \sum_i p_i \int_{Y} \log 1/\|Dg_{i}(y)^{-1}\| \nu_{\{p_i\}}(dy). \] 
Hence, 
\[ \limsup_{n \to \infty} \frac{\log G_n (x)}{\log F_n (x)} \le \beta, \ \textup{ $\mu_{\{p_i\}}$-a.e.$x$.}  \]
Essentially the same argument as above is done in the proof of \cite[Theorem 1.1]{O16}. 

Let $\epsilon > 0$. 
Then, we have that for $\mu_{\{p_i\}}$-a.e.$x$, it holds that for sufficiently large $n$, $F_n (x) < 1$ and $G_n (x) < 1$, and hence,  
\[ G_n (x) \ge F_n (x)^{\beta + \epsilon}. \]

Let 
\[ F_{0,n}(x) := \left| f_{I_1(x)} \circ \cdots \circ f_{I_n(x)}(\textup{fix}(f_0)) - x \right|. \]
\[ F_{N-1,n}(x) := \left| x - f_{I_1(x)} \circ \cdots \circ f_{I_n(x)}(\textup{fix}(f_{N-1})) \right|. \] 
\[ G_{0,n}(x) := \left| g_{I_1(x)} \circ \cdots \circ g_{I_n(x)}(\textup{fix}(g_0)) - \varphi(x) \right| \]
\[= \left| \varphi(f_{I_1(x)} \circ \cdots \circ f_{I_n(x)}(\textup{fix}(f_0))) - \varphi(x) \right|. \]
\[ G_{N-1,n}(x) := \left| \varphi(x) - g_{I_1(x)} \circ \cdots \circ g_{I_n(x)}(\textup{fix}(g_{N-1})) \right| \]
\[= \left| \varphi(f_{I_1(x)} \circ \cdots \circ f_{I_n(x)}(\textup{fix}(f_{N-1}))) - \varphi(x) \right|. \] 
Then, by Assumption \ref{ass-diam}, 
\[ 2\max\{G_{0,n}(x), G_{N-1,n}(x)\} \ge G_n (x) \ge (c/2)^{\beta + \epsilon} (F_{0,n}(x) + F_{N-1,n}(x))^{\beta + \epsilon}.  \]

Hence, there exists a constant $ \widetilde c > 0$ such that 
\[ \max\left\{\limsup_{n \to \infty} \frac{G_{0,n}(x)}{F_{0,n}(x)^{\beta + \epsilon}}, \limsup_{n \to \infty} \frac{G_{N-1,n}(x)}{F_{N-1,n}(x)^{\beta + \epsilon}} \right\} \ge \widetilde c, \ \textup{$\mu_{\{p_i\}}$-a.e.$x$.}  \]
This completes the proof of assertion (ii). 

We second show assertion (i).  
Let 
\[ n(x, x^{\prime}) := \min\left\{k : I_k (x) \ne I_k(x^{\prime}) \right\}, \ x, x^{\prime} \in \bigcup_{i \in I} \widetilde X(i). \]

We compare $N^{-n(x, x^{\prime})}$ with $|x - x^{\prime}|$.   
By the definition of $\{I_i\}_i$, 
\[ x \in f_{I_1(x)} \circ \cdots \circ f_{I_{\lfloor n(x, x^{\prime})(1+\epsilon) \rfloor}(x)}(X),  \] 
and
\[ x^{\prime} \in X \setminus f_{I_1(x)} \circ \cdots \circ f_{I_{n(x, x^{\prime})}(x)}(X). \]
Let $\epsilon > 0$. 
Then by Assumption \ref{ass-reg}, 
\[ |x - x^{\prime}| \ge c \ \textup{diam}\left(f_{I_1(x)} \circ \cdots \circ f_{I_{\lfloor n(x, x^{\prime})(1+\epsilon) \rfloor}(x)}(X)\right). \]

Now we give an upper bound for $|\varphi(x) - \varphi(x^{\prime})|$.  
Due to \eqref{conjugate}, 
both of the $\varphi(x)$ and $\varphi(x^{\prime})$ are contained in 
\[ \varphi\left(f_{I_1(x)} \circ \cdots \circ f_{I_{n(x, x^{\prime})-1}(x)}(X)\right) = g_{I_1(x)} \circ \cdots \circ g_{I_{n(x, x^{\prime})-1}(x)}(\varphi(X))  \] 
\[ \subset g_{I_1(x)} \circ \cdots \circ g_{I_{n(x, x^{\prime})-1}(x)}(Y). \]

Therefore, 
\[ |\varphi(x) - \varphi(x^{\prime})| \le \textup{diam}\left(g_{I_1(x)} \circ \cdots \circ g_{I_{n(x, x^{\prime})-1}(x)}(Y)\right). \]

By noting \eqref{tx-f-estimate}, \eqref{tx-g-estimate}, $\mu_{\{p_i\}_i}\left( \bigcup_{i \in I} \widetilde X(i) \right) = 1$, and the fact that $\varphi$ is continuous, 
it suffices to show that 
\[ \liminf_{n \to \infty} \frac{\textup{diam}\left(g_{I_1(x)} \circ \cdots \circ g_{I_{\lfloor n(1+\epsilon) \rfloor}(x)}(Y)\right)}{\textup{diam}\left(f_{I_1(x)} \circ \cdots \circ f_{I_{n-1}(x)}(X)\right)^{\alpha}} = 0, \  \textup{$\mu_{\{p_i\}}$-a.s.$x$.}  \]

By \cite[Lemma 3.1]{O16}, 
we have that 
\[ \frac{1}{\| (Df_{I_1(x)}(Tx))^{-1}\|} + o_n (1) \le \frac{\textup{diam}\left(f_{I_1(x)} \circ \cdots \circ f_{I_{n}(x)}(X)\right)}{\textup{diam}\left(f_{I_1(x)} \circ \cdots \circ f_{I_{n-1}(x)}(X)\right)} \le \| Df_{I_1(x)}(Tx)\| + o_n (1), \]
and 
\[ \frac{1}{\| (Dg_{I_1(x)}(\varphi(Tx)))^{-1}\|} + o_n (1) \le \frac{\textup{diam}\left(g_{I_1(x)} \circ \cdots \circ g_{I_{n}(x)}(Y)\right)}{\textup{diam}\left(g_{I_1(x)} \circ \cdots \circ g_{I_{n-1}(x)}(Y)\right)} \le \| Dg_{I_1(x)}(\varphi(Tx))\| + o_n (1),  \] 
which correspond to \eqref{tx-f-estimate} and \eqref{tx-g-estimate}, respectively. 

In the same manner as in the proof of assertion (i), we have that 
\[ \limsup_{n \to \infty} \frac{\log \textup{diam}\left(g_{I_1(x)} \circ \cdots \circ g_{I_{\lfloor n(1+\epsilon) \rfloor}(x)}(Y)\right)}{n } \]
\[\le (1+ \epsilon) \left(\sum_i p_i \int_Y \log \| Dg_i (y)\| \nu_{\{p_i\}}(dy) \right), \ \ \textup{$\mu_{\{p_i\}}$-a.s.$x$.}  \]
and, 
\[ \liminf_{n \to \infty} \frac{\log \textup{diam}\left(f_{I_1(x)} \circ \cdots \circ f_{I_{n-1}(x)}(X)\right)}{n} \ge (1-\epsilon) \left(\sum_i p_i \int_X \log 1/\| Df_i (y)^{-1}\| \mu_{\{p_i\}}(dy)\right) \]
\[ \ge \frac{1-2\epsilon}{\alpha(1-\epsilon/2)} \left(\sum_i p_i \int_Y \log \| Dg_i (y)\| \nu_{\{p_i\}}(dy) \right), \ \  \textup{$\mu_{\{p_i\}}$-a.s.$x$,}  \]
where we have used the assumption of (i). 

We remark that if we take sufficiently small $\epsilon > 0$,
\[ \frac{1-2\epsilon}{1-\epsilon/2} < 1+\epsilon. \]
By using this and 
\[ \sum_i p_i \int_Y \log \| Dg_i (y)\| \nu_{\{p_i\}}(dy) < 0, \]
we have assertion (i). 
\end{proof} 


\section{Examples for existence and uniqueness} 

\subsection{Examples for the case that $g_i$ does not depend on $x$}

\begin{Exa}[The topological structure of $X$ is not given]
(i) Assume that $I$ contains at least two distinct points $\{i_0, i_1\}$. 
If $f_i : X_i \to X$ is an identity map for each $i \in I$,   
then, $A = X_{i_0} \cap X_{i_1}$.   
If $X_{i_0} \cap X_{i_1} \ne \emptyset$ and $Y$ contains at least two distinct points $\{y_0, y_1\}$ and $g_{i_k}  \equiv y_k$, $k = 0,1$, 
then, Assumption \ref{ass-new-compat} fails and hence there is no solutions for \eqref{conjugate}. 
This is also an example such that the compatibility conditions in \cite[Definition 1]{SB17} fails.\\
(ii) In (i) above, it is crucial to assume that $X_0 \cap X_1 \ne \emptyset$. 
Indeed, if $I = \{0,1\}$, $X_0 = S \subset X$ and $X_1 = X \setminus S$, $S$ and $X \setminus S$ are both non-empty,  $f_i$ are identity maps, $Y$ contains at least two distinct points $\{y_0, y_1\}$, and $g_{k}  \equiv y_k$, $k = 0,1$,  
then, a function $\varphi : X \to Y$ such that $\varphi = y_0$ on $S$ and $\varphi = y_1$ on $X \setminus S$ gives a solution for \eqref{conjugate}. 
\end{Exa} 

We now give examples for $X = Y = [0,1]$. 
Here and henceforth, we always give $[0,1]$ the topology induced by the Euclid metric. 
First we consider  the case that $A = \{0,1\}$.  

\begin{Prop}[Existence and continuity]\label{interval-exist-con}
Let $X = Y = [0,1]$. 
Assume that each $f_i$ is a weak contraction in the sense of Matkowski on $[0,1]$ satisfying that 
\[ f_0(0) = 0, \ f_{N-1}(1) = 1, \ f_{i-1}(1) = f_i(0), \ 1 \le i \le N-1. \]
\[ f_{i}(0) \le f_i (x) \le f_i (1), \ \ x \in [0,1]. \]
\[ f_0 (x) > 0, f_{N-1}(x) < 1, \ \ x \in (0,1). \]
Assume that each $g_i$ is a strictly increasing weak contraction in the sense of Matkowski on $[0,1]$ such that 
\[ g_0(0) = 0, \ g_{N-1}(1) = 1, \ g_{i-1}(1) = g_i(0), \ 1 \le i \le N-1. \]
Then, the unique solution of \eqref{conjugate} exists and is continuous.  
\end{Prop}

\begin{Rem}\label{Ruan-interval}
We do {\it not} assume that each $f_i$ is injective. Therefore, the method taken in \cite{R10} is not applicable to  this case, however, in the following proof we heavily depend on the order structure of $X = [0,1]$. 
\end{Rem}

\begin{proof}
Assumptions \ref{ass-f}, \ref{ass-new-compat} and \ref{ass-g-wc} hold. 
Hence by Theorem \ref{non-injective}, the unique solution of \eqref{conjugate} exists. 
Now we show the continuity of the solution. 

First we recall the following simple assertion.  
\begin{Lem}\label{modification}
Let $D \subset [0,1]$ be a dense subset of $[0,1]$ and $\phi : D \to [0,1]$ is increasing. 
Let 
\[ \phi^{D}_{+}(x) := \lim_{y \to x, y \in D \cap (x,1]} \phi(y), \  x \in [0,1), \ \phi^{D}_{+}(1) := 1. \] 
\[\phi^{D}_{-}(x) := \lim_{y \to x, y \in D \cap [0,x)} \phi(y), \  x \in (0,1], \ \phi^{D}_{-}(0) := 0.  \]
Then, $\phi^{D}_{+}$ and $\phi^{D}_{-}$ are right and left continuous, respectively. 
\end{Lem} 

It holds that $\widetilde A$ is dense in $[0,1]$ and $\varphi$ is increasing on $\widetilde A$. 

Since all $f_i$, $g_i$, $0 \le i \le N-1$, are continuous, 
the left and right continuous modifications of the solution of  \eqref{conjugate} is also the solution of \eqref{conjugate}. 
Hence, by the uniqueness of \eqref{conjugate}, these functions are identical with each other.   
By this and Lemma \ref{modification}, the unique solution is left and right continuous, and hence, continuous.   
\end{proof}

In the following example, we deal with the {\it Minkowski question-mark function}, which is often denoted by $?(x)$. 
It was introduced by Minkowski \cite{Min1904}  in 1904 and has been closely connected with continued fractions, number theory and dynamical systems. 
It is defined by 
\[ ?(x) = \sum_{k=1}^{\infty} (-1)^{k-1} 2^{1-(n_1 + \cdots + n_k)}, \]
if $x \in [0,1]$ has the following continued fraction expansion: $$x = [n_1, n_2, \dots] = \dfrac{1}{n_1 + \frac{1}{n_2 + \cdots}}. $$
There is another way of (equivalent) definition which uses the Farey sequence. 
De Rham \cite{dR57} regarded $?(x)$ as a solution of a functional equation.

Singularity of the Minkowski question-mark function was considered by Denjoy \cite{D38} and Salem \cite{Sa43}. 
Some analytic properties of the Minkowski question-mark function were considered by Viader, Paradis and Bibiloni \cite{VPB98}, Kesseb\"ohmer and Stratmann \cite{KS08}, Jordan and Sahlsten \cite{JS13}, etc. 
The inverse function of the Minkowski question-mark function is called the {\it Conway box function}.
Analytic properties of the inverse  were recently considered by Mantica and Totik \cite{MT18+}. 

\begin{Exa}[Equation driven by weak contractions]\label{Minkowski}
Let $I = \{0,1\}$, $X = Y = [0,1]$. 
Let $f_0 (x) = x/2$, $f_1 (x) = (x+1)/2$, $g_0 (y) = y/(y+1)$ and $g_1 (y) = 1/(2-y)$.  
Neither $g_0$ or $g_1$ is a contraction, but both of them are weak contractions in the sense of Browder. 
Then,\\ 
(i) A unique solution $\varphi$ of \eqref{conjugate} is the Conway box function, that is, the inverse function of the Minkowski question-mark function.\\ 
(ii) For any dyadic rational $x$ on $[0,1]$ and $a > 0$,  
\begin{equation}\label{M-explode-dyadic}  
\limsup_{y \to x, y > x} \frac{f(y) - f(x)}{|x - y|^a} = \limsup_{y \to x, y < x} \frac{f(x) - f(y)}{|x - y|^a} =  +\infty. 
\end{equation} 
Let $f_0 (x) = x/(x+1)$, $f_1 (x) = 1/(2-x)$, $g_0 (y) = y/2$, and $g_1 (y) = (y+1)/2$.    
Then,  a unique solution $\varphi$ of \eqref{conjugate} is the Minkowski question-mark function. 
\end{Exa}

The framework of de Rham curves contains the Minkowski question-mark function. 

\begin{Exa}[De Rham curves] 
Let $N \ge 2$. 
Let $I = \{0,1,\dots,N-1\}$.  
Let $X = [0,1]$. 
Let $f_i (x) = (x+i)/N$.  
Let $Y$ be a complete metric space and $g_i : Y \to Y$ be a weak contraction for each $i$.  
Denote a unique fixed point of $g_i$ by $\textup{fix}(g_i)$.  
Assume that for each $1 \le i \le N-1$,  
\[ g_{i-1}(\textup{fix}(g_{N-1})) = g_i (\textup{fix}(g_0)). \] 

In this case \eqref{conjugate} is called a de Rham curve. 
This is initially considered by de Rham \cite{dR57}. 
\cite{dR57} deals with the case that $N=2$ and regard the Minkowski question-mark function as a solution of a specific functional equation. 
The above setting, which is a little generalization of the framework of \cite{dR57}, is considered by Hata \cite[Sections 6 and 7]{Ha85}. 

Assumptions \ref{ass-f}, \ref{ass-new-compat}, \ref{ass-g-wc} and \ref{ass-A} are satisfied, and hence by Theorem \ref{exist-unique}, there exists a unique solution of \eqref{conjugate}. 
The case that each $f_i$ is a more general function is considered by \cite[Theorem 6.5]{Ha85}. 
Multifractal analysis for the solution is considered by \cite{BKK18}.  
\end{Exa}

Second we consider the case that $A \ne \{0,1\}$. 
This case is more difficult. 
Here we only give some examples in which the the solution of \eqref{conjugate} do {\it not} exist. 

\begin{Exa}[Between two IFSs with overlaps]\label{exa-overlap}
Let $I = \{0,1\}$, $X = Y = [0,1]$. 
Assume that $f_0 (x) = ax$, $f_1 (x) = ax + 1-a$, $g_0 (x) = bx$ and $g_1 (x) = bx + 1-b$ for some $a, b \in [1/2, 1)$.    
Then,\\
(i) Let $a = 3/4$ and $b = 1/2$.  
Then, $A = [0,1]$ and there is no solution for \eqref{conjugate}. 
Assume that $\varphi$ is a solution for \eqref{conjugate}.  
Since $\varphi(3/4) = \varphi(f_0 (1)) =  g_0 (\varphi(1))$,
it holds that  $\varphi(3/4) = \varphi(f_1 (2/3)) = g_1 (\varphi(2/3))$.  
Hence, $g_0 (\varphi(1)) = g_1 (\varphi(2/3)) = 1/2$.
By the definitions of $g_i$, $\varphi(2/3) = 0$.
On the other hand, 
\[ \varphi(2/3) = \varphi(f_1 (5/9)) = g_1 (\varphi(5/9)) \in [1/2, 1]. \] 
This leads a contradiction. \\  
(ii) Let $a = 1/2$ and $b = 3/4$.   
Then, $A = \{0,1\}$ and the compatibility condition of \cite[Definition 1]{SB17} fails.\\
(iii) Let $a = 2/3$ and $b = 3/4$.   
Assume that $\varphi$ is a solution for \eqref{conjugate}.  
Then, there exist two candidates for the value of $\varphi(1/2)$. 
This leads a contradiction.
\end{Exa}

\subsection{Examples for the case that $g_i$ depends on $x$}

\begin{Exa}[A case that a unique solution of \eqref{conjugate} is not continuous]
Fix $i \in I$. 
Assume that $X$ is a topological space and  $f_i$ is continuous. 
Assume that there exists a continuous function $G : Y \times Y \to Y$ such that $G_i  : X \times Y \to Y$ does not depend on $y$ and is not continuous as a function on $X$, 
where we let $G_i (x, y) = G(g_i(x,y),y)$.  
Then, under Assumptions \ref{ass-f}, \ref{ass-new-compat}, \ref{ass-g-wc}  and  \ref{ass-A}, 
a unique solution of \eqref{conjugate} is {\it not} continuous.

For example, this is applicable to the case that $X = [0,1]$, $f_0 (x) = x/2$, $f_1 (x) = (x+1)/2$,  
$Y = \mathbb{R}$, $\widetilde  g_0 (x) = px$, $\widetilde g_1(x) = (1-p) x + p$, $p \in (0,1)$, $h : X \to Y$ is a non-continuous function, 
$g_i (x,y) = h(x) + \widetilde g_i (y)$ and $G(y_1, y_2) := y_1 - \widetilde g_i (y_2)$. 
\end{Exa}

\begin{Exa}[Sierpi\'nski gasket]\label{SG-1} 
Let $V_0 = \{q_0, q_1, q_2\} \subset \mathbb{R}^2$ be a set of three points such that $\|q_i - q_j \| = 1, i \ne j$.   
Let $f_i (x) = (x+q_i)/2$, $i = 0,1,2$, $x \in \mathbb{R}^2$. 
Let $X$ be a unique compact subset of $\mathbb{R}^2$ such that $X = \cup_{i \in I} f_i(X)$.   
For $i \in I$ and $x \in X$, 
let $g_i (x, \cdot)$ be weak contractions on $\mathbb{R}$ such that 
\begin{equation}\label{g-fix} 
g_i (q_j, \textup{fix}(g_j)) = g_j (q_i, \textup{fix}(g_i))  \textup{ for each $(i,j)$}. 
\end{equation}     
$A = V_0$ and Assumptions \ref{ass-f}, \ref{ass-new-compat}, \ref{ass-g-wc}  and  \ref{ass-A} hold.  
Therefore by Theorem \ref{exist-unique}, there exists a unique solution of \eqref{conjugate}. 
This case is considered by \cite{CKO08}, \cite{R10} and \cite{RR11}. 
We remark that by \eqref{g-fix}, 
if each $g_i (\cdot) = g_i (x, \cdot)$ is linear, then, the solution is quite limited. 
\end{Exa}

\subsection{Example for the case that infinitely many solutions exist} 

In this subsection, we consider the case that at least one $g_i$ is not weakly contractive, 
that is, Assumptions \ref{ass-f} and \ref{ass-new-compat} hold, but Assumption \ref{ass-g-wc} {\it fails}.  

Before we delve into the main result, we give an instructive example. 

\begin{Exa}[Example for the case that the uniqueness fails; Cf. \cite{O14T}]
Consider \eqref{conjugate} for the case that $I = \{0,1\}, \ X = Y = [0,1], f_i (x) = (x+i)/2, \ g_i (y) = \Phi(A_{u,i} ; y), \ i = 0,1$,
where we let 
\[ \Phi(A ; z) = \frac{az + b}{cz + d} \, \text{ for } A = \begin{pmatrix} a & b  \\ c & d \end{pmatrix}, \text{ and, } \]
\[ A_{u, 0} = \begin{pmatrix} x_{u} & 0  \\ -u^{2}x_{u}^{2} & 1 \end{pmatrix}, \ A_{u, 1} = \begin{pmatrix} 0 & x_{u} \\ -u^{2}x_{u}^{2} & 1 - u^{2}x_{u}^{2} \end{pmatrix}, \ x_{u} = \frac{2}{1+\sqrt{1+8u^{2}}}. \] 
In this case, by \cite[Proposition 3.4]{O14T}, it holds that for any dyadic rational $x$, 
\[ \lim_{y \to x, y < x}  \varphi_u (y)  < \varphi_u (x). \]
This example is different from examples in Serpa-Buescu \cite{SB15c}. 
Here $g_i (y)$ is {\it not} affine as a function of $y$. 
If $u > \sqrt{3}$, then, $g_1$ is {\it not} a weak contraction, the Lipschitz constant of $g_1$ is strictly larger than $1$ and $g_1$ has two fixed points $y_{10} < y_{11}$.   
Obviously $y_{11} = 1$. 
There are two possibilities for the value of $\varphi(1/2)$.  
There are two solutions $\varphi_0$ and $\varphi_1$ of \eqref{conjugate} satisfying that $\varphi_i (1/2) = g_0 (y_{1i}), i = 0,1$.   
Then, the right-continuous modification of $\varphi_0$ is equal to $\varphi_1$ on the set of dyadic rationals, and, 
the left-continuous modification of $\varphi_1$ is equal to $\varphi_0$ on the set of dyadic rationals.    
\end{Exa} 

\begin{Thm}\label{infinite-sol}
Let $X =  Y = [0,1]$ and $ f_i (x) = (x+i)/2, i = 0,1$.  
Then there exist strictly increasing functions $g_0$ and $g_1$  on $[0,1]$ such that 
Assumption \ref{ass-A} holds and furthermore there exist infinitely many solutions for \eqref{conjugate}.  
\end{Thm}

\begin{proof}
Recall that $M$ is the {\it inverse} function of Minkowski's question-mark function appearing in Example \ref{Minkowski}.      
Let $g_0 (x) := M(x)/2$ and $g_{1, 0}(x) := (x+1)/2$. 
We define strictly increasing functions $g_{1,n} : [0,1] \to [0,1]$, $n \ge 1$, in the following manner.      
We have that $g_{1, 0} \circ g_0(0) = 1/2$ and $g_{1, 0} \circ g_0(1) = 3/4$. 
By the intermediate value theorem, 
there exists at least one point $y_0 \in (1/2, 3/4)$ such that $g_{1, 0} \circ g_0(y_0) = x_0$. 
Then there exists a strictly increasing continuous function $g_{1,1}$ such that 
\[ \sup_{x \in [0,1]} |g_{1,1}(x) - g_{1,0}(x)| \le 1/8, \]
\[ g_{1,1}(x) = g_{1,0}(x), \  \ |x - y_0| > \frac{\min\{y_0 - 1/2, 3/4-y_0\}}{8}.  \]
and, for a dyadic rational $x_1 \in (1/2, 3/4)$, $g_{1, 1} \circ g_0(x_1) = x_1$. 
This is possible because we can take a sequence $\{z_n\}_n$ of dyadic rationals converging to $y_0$, and due to the continuity of $g_0$, $g_0 (z_n) \to g_0 (y_0)$ as $n \to \infty$. 

Now by \eqref{M-explode-dyadic}, there exists a point $y_2 > x_1$ such that  $g_{1, 1} \circ g_0(y_2) = y_2$. 
Then there exists a strictly increasing continuous function $g_{1,2}$ such that 
\[ \sup_{x \in [0,1]} |g_{1,2}(x) - g_{1,1}(x)| \le 1/16, \]
\[ g_{1,2}(x) = g_{1,1}(x), \  \ |x - y_2| > \frac{\min\{y_2 - x_1, 3/4 - y_2\}}{16}.  \]
and, there exists at least one dyadic rational $x_2 \in (x_1, 3/4)$ such that $g_{1, 2} \circ g_0(x_2) = x_2$. 
By repeating this procedure, 
we can define the limit $g_1 := \lim_{n \to \infty} g_{1,n}$. 
Then, $g_1$ is an increasing continuous function such that 
$g_{1} \circ g_0(x_n) = x_n$ for any $n \ge 1$, and, $g_1(0) = 1/2$ and $g_1(1) = 1$.

We remark that 
\[ C := \bigcup_{n \ge 1} \bigcup_{i_1, \dots, i_n \in I} f_{i_1} \circ \cdots \circ f_{i_n}(\{1/3, 2/3\})\] 
satisfies that \[ \bigcup_{i} f_i^{-1}(C) \subset C,\] 
so the values of each solution of \eqref{conjugate} on the set $C$ is determined independently on the value of $\varphi$ on $X \setminus C$.  
If we give a value of $\varphi(1/3)$, then, all the values of $\varphi$ on $C$ are uniquely determined by \eqref{conjugate}.    
Since $1/3$ is a fixed point of $f_0 \circ f_1$ on $[0,1]$, 
\[ \varphi(1/3) = \varphi(f_0 \circ f_1(1/3)) = g_0 \circ g_1(\varphi(1/3)). \]
Hence $\varphi(1/3)$ is a fixed point of $ g_0 \circ g_1$.   
By our choice of $g_0$ and $g_1$, 
there exist at least countably many fixed points of  $g_0 \circ g_1$ in $[0,1]$, and hence there are at least countably many candidates of $\varphi(1/3)$, and each value can be taken.  

It will be easy to see that Assumption \ref{ass-A} holds.   
\end{proof}

\begin{Rem}
The idea of introducing $C$ in the proof is similar to that in introducing Assumption \ref{ass-A}. 
\end{Rem}


\section{Examples for regularity}

\subsection{The case that $X = [0,1]$} 

\begin{Exa}[linear case]\label{linear-dR}  
Let $X = Y = [0,1]$. 
Let $f_0 (x) = x/2$ and $f_1(x) = (x+1)/2$. 
Let $a \in (0,1)$. 
Let $g_0 (y) = ay$ and $g_1(y) = (1-a)y + a$. 
Then, the solution $\varphi$ of \eqref{conjugate} is Legesgue's singular function. 
(We remark that by a direct use of \eqref{conjugate}, it holds that  
$\varphi$ is $( -\log\max\{a, 1-a\}/\log 2)$-H\"older continuous.) 
Let $0 < p < 1$ and let 
\[ \alpha > \frac{-p \log a - (1-p) \log (1-a)}{\log 2}.  \] 
Then, by Theorem \ref{Thm-dim-Haus-2}, 
$(\alpha, x, +\infty)$ holds for 
$\mu_{(p, 1-p)}$-a.e. $x$. 
\end{Exa}

The Minkowski question-mark function has its derivative and is zero at every dyadic rational.  
Now we give an another example of singular function whose derivative at every dyadic rational is zero. 

\begin{Exa}[The derivative at each dyadic point vanishes]   
Let $X = Y = [0,1]$. 
Let $f_0 (x) = x/2$ and $f_1(x) = (x+1)/2$.  
Let $\widetilde g_0 (y) = y/2$ on $[0,1/2]$ and $\widetilde g_0 (y) = y - 1/4$ on $[1/2, 1]$. 
Let $\widetilde g_1(y) = (y+3)/4$.  
Let $g_0, g_1 : [0,1] \to [0,1]$ be continuous strictly increasing functions such that \\ 
(i) $0 = g_0(0) < 3/4 < g_0(1) = g_1(0) < g_1(1) = 1$.\\
(ii) $g_1$ is linear and $g_0$ is piecewise linear.\\   
(iii) $g_0 < \widetilde g_0$ on $[0, 1/2]$ and $g_0 > \widetilde g_0$ on $[1/2, 1]$. \\
(iv) $g_0$ is differentiable on the open interval $(0, 1/4)$ and $0 < g_0^{\prime}(y) < 1/2$ holds for every $y$.\\ 
(v) For some $\epsilon > 0$ which will be specified later,
\[ \sup_{y \in [0,1]} \left| (\widetilde g_0)^{\prime}(y) - g_0^{\prime}(y) \right| +  \sup_{y \in [0,1]} \left| (\widetilde g_1)^{\prime}(y) - g_1^{\prime}(y) \right| < \epsilon. \]
(Such $g_0$ and $g_1$ exist for any $\epsilon > 0$.)   
We will show that under this setting, 
the solution \eqref{conjugate} has derivative zero at every dyadic rational.
\end{Exa} 

\begin{proof}  
Let $0 < p < 1$  which will be specified later.    
Let $\mu_{p, 1-p}$ and $\nu_{p, 1-p}$ be the probability measures on $[0,1]$ 
such that 
\[ \mu_{p, 1-p} =  p \mu_{p, 1-p} \circ f_0^{-1} + (1-p) \mu_{p, 1-p} \circ f_1^{-1}, \nu_{p, 1-p} = p \nu_{p, 1-p} \circ g_0^{-1} + (1-p) \nu_{p, 1-p} \circ g_1^{-1}.\]  
In this case, it is easy to see that the derivative of the solution $\varphi$ of \eqref{conjugate} at each dyadic point exists and is zero. 

By Theorem \ref{Thm-dim-Haus-2}, it suffices to show that for some $p \in (0,1)$ and $\epsilon > 0$, 
\[ \int_{[0,1]} -p\log \left| g_0^{\prime}(y)\right| - (1-p) \log \left| g_1^{\prime}(y)\right| \nu_{p, 1-p}(dy) < \log 2. \]

We remark that $\nu_{p, 1-p}$ depend on the choice of $(p, g_0, g_1)$ so hereafter we write the measure as $\nu_{p, g_0, g_1}$. 
\[  \int_{[0,1]} -p\log \left| (\widetilde g_0)^{\prime}(y)\right| - (1-p) \log \left|(\widetilde g_1)^{\prime}(y)\right| \nu_{p, g_0, g_1}(dy)  \]
\[ = \left(p \nu_{p, g_0, g_1}([0, 1/2]) + 2(1-p)\right) \log 2 \]
(use $g_0 (1) > \widetilde g_0 (1) = 3/4$ and $g_0 \circ g_0 (1) > \widetilde g_0  \circ \widetilde g_0 (1)$.)
\[ \le (p \nu_{p, g_0, g_1}\left( g_0 \circ  g_0 ([0, 1]) \right) + 2(1-p)) \log 2  = (p^3 - 2p + 2) \log 2. \]
If we let $p = 3/4$, then, $p^3 - 2p + 2 < 1$. 
Hence if we take sufficiently small $\epsilon > 0$, 
\[ \int_{[0,1]} -p\log \left| g_0^{\prime}(y)\right| - (1-p) \log \left| g_1^{\prime}(y)\right| \nu_{p, 1-p}(dy) < \log 2. \]
\end{proof}

In some cases, computation for 
\[ \sum_{i \in I} p_i \int_{[0,1]} \log (1/g^{\prime}_i (y)) \nu_{\{p_i\}}(dy)\] 
is hard, because $g_i$ can be a non-linear function and hence the integral depends essentially on $\nu_{\{p_i\}}$.  

\begin{Lem}\label{weak-con}
Let $(Z, d_Z)$ be a compact metric space and $\{h_i\}_i$ be a collection of weak contractions in the sense of Browder on $Z$ satisfying that 
\[ Z = \bigcup_{i \in I} h_i (Z).\]     
Assume that \[ \eta_{\{p_i\}} = \sum_{i \in I} p_i \eta_{\{p_i\}} \circ h_i^{-1}.\]    
Then, for each fixed $i$, as $p_i \to 1$, $\eta_{\{p_i\}}$ converges weakly to the delta measure on $\textup{fix}(h_i)$.  
\end{Lem}

\begin{proof}
Let $F$ be a real continuous bounded function on $Z$.  
Let $\epsilon > 0$. 
Then, by using the fact that $h_i$ is a weak contraction, we can take a sufficiently large $n$ such that  
\[ \sup_{z \in Z} \left|F \circ h_i^{n}(z) - F(\textup{fix}(h_i))\right| \le \epsilon. \]
By the definition of $\eta_{\{p_i\}}$,  
\[ \left| \int_Z F d\eta_{\{p_i\}}  - p_i^n \int_{Z} F \circ h_i^{n} d\eta_{\{p_i\}} \right| \le (1 - p_i^n) \sup_{z \in Z} |F(z)|. \]
Hence for every $p_i$ sufficiently close to $1$, 
\[ \left| \int_Z F d\eta_{\{p_i\}}  - F(\textup{fix}(h_i)) \right| \le 3\epsilon.\] 
\end{proof} 

\begin{Prop}\label{fix-deri}
Let $I = \{0,1,\dots, N-1\}$.  
Let $X = Y = [0,1]$. 
Assume that $f_i$ and $g_i$ are $C^1$ functions on $[0,1]$ such that 
\[ 0 = f_0(0) = g_0(0)  < f_1(1) = g_1(1) = 1, \]
\[ f_{i-1}(1) = f_i (0), \ g_{i-1}(1) = g_i (0), \ \ 1 \le i \le N-1, \]
\[ 0 < \min_{i \in I} \inf_{z \in [0,1]} \min\{f_i^{\prime}(z), g_i^{\prime}(z)\} \le \max_{i \in I} \sup_{z \in [0,1]} \max\{f_i^{\prime}(z), g_i^{\prime}(z)\}  < 1. \]
Fix $i \in I$. 
 Let $\alpha > \log \left(1/|g_i^{\prime}(\textup{fix}(g_i))| \right) /\log (1/|f_i^{\prime}(\textup{fix}(f_i))|)$.
Then, $(\alpha, U, +\infty)$ holds for every open set $U$. 
\end{Prop}

\begin{proof}
By Lemma \ref{weak-con},  
\[ \lim_{p_i \to 1} \frac{\sum_{i \in I} p_i  \int \log (1/|g_i^{\prime} (y)|) \nu_{\{p_i\}}(dy)}{\sum_{i \in I} p_i  \int \log 1/|f_i^{\prime}(x)| \mu_{\{p_i\}}(dx)}  = \frac{\log 1/ \left|g_i^{\prime}(\textup{fix}(g_i)) \right|}{\log 1/ \left|f_i^{\prime}(\textup{fix}(f_i)) \right|} < \alpha. \]
Now use Theorem \ref{Thm-dim-Haus-2} for an appropriate $(p_{0}, \dots p_{N-1})$.  
Then, $(\alpha, x, +\infty)$ holds for $\mu_{\{p_i\}}$-a.e.$x$.  
We remark that $\mu_{\{p_i\}}(U) > 0$, and then the assertion follows.   
\end{proof}

\begin{Prop}\label{dist-func}
Let $I = \{0,1,\dots, N-1\}$.  
Let $X = Y = [0,1]$. 
Assume that $f_i, g_i$ are strictly increasing continuous functions on $[0,1]$ such that 
\[ 0 = f_0(0) = g_0(0)  < f_1(1) = g_1(1) = 1, \]
\[ f_{i-1}(1) = f_i (0), \ g_{i-1}(1) = g_i (0), \ \ 1 \le i \le N-1. \]
and each $g_i$ is a weak contraction in the sense of Matkowski. 
In this case a unique solution $\varphi$ of \eqref{conjugate} is a strictly increasing continuous function on $[0,1]$ such that $\varphi(0) = 0$ and $\varphi(1) = 1$.  
Let $\mu_{\varphi}$ be a unique probability measure whose distribution function is $\varphi$, that is, 
\[ \mu_{\varphi}\left((a, b] \right) = \varphi(b) - \varphi(a), \ \ 0 \le a < b. \]   
Assume $\alpha > \dim_H \mu_{\varphi}$.  
Then, $(\alpha, U, +\infty)$ holds for every open interval $U$.  
\end{Prop}

\begin{proof}
We remark that $\textup{supp}(\mu_{\varphi}) = [0,1]$ and 
\[ \dim_H \mu_{\varphi} = \inf\{\dim_H K : \mu_{\varphi}(K) = 1 \}. \]
Now the proof is easy to see. 
Assume that there exists an open interval $U$ such that $(\alpha, U, +\infty)$ fails. 
Then, there exists a constant $C$ such that for any $x_1, x_2 \in U$, 
\[  |\varphi(x_1)- \varphi(x_2)| \le C |x_1 - x_2|^{\alpha}. \]
Since $\alpha > \dim_H \mu_{\varphi}$, 
we can take $K \subset [0,1]$ such that $\dim_H K < \alpha$ and $\mu_{\varphi}(K) = 1$.   
Let $\beta \in (\dim_H K,  \alpha)$.  
Let $\delta > 0$ be an arbitrarily taken number. 
Then there exist a countably infinitely many number of pairs $a_i < b_i$ 
such that $\sup_i (b_i - a_i) < \delta$
and 
$K \cap U \subset \cup_i (a_i, b_i)$ and $\sum_i (b_i - a_i)^{\beta} \le 1$. 
Furthermore, by $\textup{supp}(\mu_{\varphi}) = [0,1]$, 
\[ 0 < \mu_{\varphi}(K \cap U) \le \sum_i \varphi(b_i) - \varphi(a_i) \le  C \sum_{i} (b_i - a_i)^{\alpha} \le C \delta^{\alpha - \beta}. \]
This cannot hold if $\delta$ is sufficiently small and contradicts the arbitrariness of $\delta$.    
\end{proof}

\begin{Exa}[The assumption of Proposition \ref{fix-deri} is not a necessary condition]  
Let $X = Y = [0,1]$. 
Let $I = \{0,1\}$. 
Let  $f_i (x) = (x+i)/2$.
Let $g_0 (y) = (5y)/(10-2y)$ and $g_1 (y) = (y+5)/(8 - 2y)$.  
Then, we can apply \cite[Theorem 1.2]{O14J} to this setting and have  $\dim_H \mu_{\varphi} < 1$ and hence, 
$\varphi$ is a singular function\footnote{A singular function is a continuous increasing function on the unit interval whose derivatives are zero at Lebesgue almost surely points.}, 
furthermore by Proposition \ref{dist-func}, $\varphi$ is not Lipschitz continuous. 
However, the assumption of Proposition \ref{fix-deri} fails. 
\end{Exa}

\begin{Rem}[Formula for $\mu_{\varphi}$] 
We assume that the solution $\varphi$ satisfies the Dini condition in \cite{FL99} and $g_i \in C^2([0,1])$ holds for each $i$.      
By the Lebesgue-Stieltjes integral, 
for every bounded Borel measurable function $F : [0,1] \to \mathbb{R}$, 
\[ \int_{[0,1]} F(x) d\mu_{\varphi}(x) = \sum_i \int_{[0,1]} g_i^{\prime}(\varphi(x)) F(f_i(x)) d\mu_{\varphi}(x).  \]

Let $h$ be a positive function on $[0,1]$ such that 
\[ h(x) = \sum_i g_i^{\prime}(g_i (\varphi(x)) ) h (f_i (x)).  \]
\[ h \circ \varphi^{-1} (x) = \sum_i g_i^{\prime}(g_i (x) ) h \circ \varphi^{-1} (g_i (x)).  \]
This is unique under the constraint that 
\[  \int_{[0, 1]}  h(x) \mu_{\varphi}(dx) = 1. \] 
See \cite[Theorem 1.1]{FL99}.  

Then, by \cite[Corollary 3.5]{FL99},  
\[ \dim_H \mu_{\varphi} = \frac{\sum_i \int_{[f_i (0), f_i (1)]} h(x) \log g_i^{\prime}(\varphi (x)) \mu_{\varphi}(dx)}{\sum_i \int_{[f_i (0), f_i (1)]}  h(x) \log f_i^{\prime}(f_i^{-1}(x)) \mu_{\varphi}(dx)} \]

Let $f_i (x) = (x+i)/N$, $i \in I$.  
Then, 
 \[ \dim_H \mu_{\varphi} = \frac{\sum_{i \in I}\int_{[0, 1]} H(g_i (y)) g^{\prime}_i (y) \log \left(1/g_i^{\prime}(g_i(y))\right) \ell(dy)}{\log N} \]  
 where $\ell$ is the Lebesgue measure on $[0,1]$,  
 \[ H(y) = \sum_{i \in I} g_i^{\prime}(g_i (y) ) H(g_i (y)),  \textup{ and } \int_{[0, 1]} H(y) \ell(dy) = 1. \] 
 It is interesting to investigate properties for $H$. 
\end{Rem}

\subsection{The case that $X$ is the two-dimensional Sierpi\'nski gasket}

Finally we deal with an example for the case that $X$ is the two-dimensional Sierpi\'nski gasket and $Y = [0,1]$.  

\begin{Exa}\label{exa-SG}
Assume that $X$ is the two-dimensional Sierpi\'nski gasket and $Y = [0,1]$.  
Here we follow the notation in Example \ref{SG-1}, however, we consider the case that each $g_i$ does not depend on $x$. 
Let $g_0 (y) = 1/(2-y) - 1/2$, $g_1(y) = (2/3)y + 1/6$ and $g_2(y) = y/(y+1) + 1/2$.   

Then, by Theorem \ref{exist-unique},   
the unique solution $\varphi$ of \eqref{conjugate} holds, 
and by adopting the method\footnote{Recall Remark \ref{Ruan-interval}} taken in \cite{R10}, 
we can show that $\varphi$ is continuous. 

Furthermore, by Theorem \ref{Thm-dim-Haus-2}, 
if \[ \beta > \frac{\frac{1}{3} \int_{[0,1]} \log \left(\frac{3}{2}(2-y+y^2)^2\right) d\nu_{(1/3, 1/3, 1/3)}(y)}{\log 2}, \] 
then, $(\beta, x, +\infty)$ holds for $\mu_{(1/3, 1/3, 1/3)}$-a.e. $x$.
\end{Exa}

\begin{Rem}
 \cite{CKO08, R10, RR11} consider fractal interpolation functions on Sierpi\'nski gasket and more generally post-critically finite self-similar sets.    
Our choices of $\{g_i\}_i$ are not considered by them, 
and it seems that their methods are not applicable to showing that $(\alpha, x, +\infty)$ holds for $\mu_{(1/3, 1/3, 1/3)}$-a.e. $x$. 
\end{Rem}

\begin{Rem}[Stability]
By arguing as in \cite[Proposition 2.7 (iii)]{O16}, 
we are able to extend Proposition \ref{fix-deri} to several cases of non-differentiability of $g_i$.  
The regularity assumptions for $f_i$ and $g_i$ are not necessarily essential. 
\end{Rem}


\section{Stability}

\begin{Def}
Let $X$ and $Y$ be two compact metric spaces. 
We let the {\it Gromov-Hausdorff distance} $d_{\textup{GH}}(X, Y)$ be the infimum of any values $d_{M}^{\textup{Haus}}(f(X), g(Y))$ 
for any compact metric space $M$ and any isometric embeddings $f : X \to M$ and $g : Y \to M$, where $d_{M}^{\textup{Haus}}$ is the Hausdorff metric on the space of all compact subsets of the space $M$. 
\end{Def}

Let $I$ be a finite set containing at least two points.   
Let $n \in \mathbb{N} \cup \{\infty\}$. 
Let $Z^{(n)}$ be a compact metric space. 
For $i \in I$ and $n \in \mathbb{N} \cup \{\infty\}$, let $h^{(n)}_j$ be a weak contraction on $Z^{(n)}$ in the sense of Browder \cite{B68}. 
Let $K^{(n)}$ be the attracter of the iterated function system $\left(Z^{(n)}, (h^{(n)}_i)_{i \in I} \right)$, that is, 
$K^{(n)}$ is the unique compact subset of $Z^{(n)}$ satisfying that 
$K^{(n)} = \bigcup_{i \in I} h^{(n)}_i \left( K^{(n)} \right)$.  

\begin{Thm}\label{GH-conv}
Let $\widetilde K^{(n)}_i, \widetilde K^{(n)}$ be subsets of $Z^{(n)}$ satisfying that 
\begin{equation}\label{approx} 
\widetilde K^{(n)} = \bigcup_{i \in I} h^{(n)}_i \left(\widetilde K^{(n)}_i \right),  
\end{equation}  
and 
\begin{equation}\label{approx-conv} 
\lim_{n \to \infty} \max_{i \in I} d^{\textup{Haus}}_{Z^{(n)}}\left(\widetilde K^{(n)}_i, \widetilde K^{(n)}\right) = 0.  
\end{equation}    
Assume that there exist compact metric spaces $\left((M^{(n)}, d_{M^{(n)}})\right)_{n \in \mathbb{N} \cup \{\infty\}}$ 
and isometries $\varphi_{n, k} : Z^{(k)} \to M^{(n)}$, $k \in \{n, \infty\}$ 
such that for each $i \in I$, \\
(a) 
\[ \lim_{n \to \infty} d^{\textup{Haus}}_{M^{(n)}} \left(\varphi_{n, n} (\widetilde K_i^{(n)}), \varphi_{n, n} (\widetilde K_i^{(n)}) \cap \varphi_{n, \infty} (Z^{(\infty)}) \right) = 0, \]
and, \\
(b) 
\[ \lim_{n \to \infty}\sup_{x \in L^{(n)}} d_{M^{(n)}}\left( \varphi_{n, n} \circ h^{(n)}_i  \circ \varphi_{n, n}^{-1}(x),  \varphi_{n, \infty} \circ h^{(\infty)}_i \circ \varphi_{n, \infty}^{-1}(x)  \right) = 0, \]
where we let 
\[ L_j^{(n)} := \varphi_{n, n} (\widetilde K_j^{(n)}) \cap \varphi_{n, \infty} (Z^{(\infty)}).\]  
Then, 
\[ \lim_{n \to \infty} d^{\textup{Haus}}_{M^{(n)}} \left( \varphi_{n, n} (\widetilde K^{(n)}),  \varphi_{n, \infty} (K^{(\infty)})  \right) = 0. \]
In particular 
\[ \lim_{n \to \infty} d^{\textup{GH}} \left( \widetilde K^{(n)},  K^{(\infty)}  \right) = 0. \]
\end{Thm}

If $K^{(n)}$ is the attracter of $\{Z^{(n)}, (h^{(n)}_i)_{i \in I}\}$, 
then, $\widetilde K^{(n)}_i = \widetilde K^{(n)} = K^{(n)}$ satisfy the assumption. 

\begin{proof}
We show this assertion by contradiction. 
Assume the conclusion fails. 
Then, by relabeling if needed, there exists $\epsilon_0 > 0$ such that for any $n$, 
\[ d^{\textup{Haus}}_{M^{(n)}} \left( \varphi_{n, n} (\widetilde K^{(n)}),  \varphi_{n, \infty} (K^{(\infty)})  \right) \ge \epsilon_0. \]

By \cite[Theorem 1 (e)]{J97}, 
for each $i \in I$, 
we can take an upper semicontinuous function $\phi_i$ such that $\phi_i (t) < t$ for any $t > 0$ and $h^{(\infty)}_i$ is $\phi_i$-contractive, 
and let  
\[ F(t) := t - \max_{i \in I} \phi_i (t), \ \ t > 0. \] 
Since $Z^{(\infty)}$ is compact, 
$\lim_{t \to \infty} F(t) = +\infty$. 
Since $F(t) > 0$ for each $t > 0$ and $F(t)$ is lower semicontinuous, and,  
\begin{equation}\label{inf-2} 
\inf_{t \ge \epsilon_0} F(t) \ge \epsilon_1 > 0. 
\end{equation} 
 
It follows that for each $j \in I$, 
\[ d^{\textup{Haus}}_{M^{(n)}} \left( \varphi_{n, n} \left( h^{(n)}_j (\widetilde K^{(n)}_j)\right),  \varphi_{n, \infty} (h^{(\infty)}_j (K^{(\infty)}) )  \right) \] 
\[\le d^{\textup{Haus}}_{M^{(n)}} \left( \varphi_{n, n} (h^{(n)}_j (\widetilde K^{(n)}_j)), \varphi_{n, n} \circ h^{(n)}_j  \circ \varphi_{n, n}^{-1} \left(L_j^{(n)}\right) \right)  \]
\[ + d^{\textup{Haus}}_{M^{(n)}} \left(\varphi_{n, n} \circ h^{(n)}_j  \circ \varphi_{n, n}^{-1} \left(L_j^{(n)}\right), \varphi_{n, \infty} \circ h^{(\infty)}_j  \circ \varphi_{n, \infty}^{-1} \left(L_j^{(n)}\right) \right) \]
\[ + d^{\textup{Haus}}_{M^{(n)}} \left(\varphi_{n, \infty} \circ h^{(\infty)}_j  \circ \varphi_{n, \infty}^{-1} \left(L_j^{(n)}\right) , \varphi_{n, \infty} \circ h^{(\infty)}_j  \circ \varphi_{n, \infty}^{-1} \left(\varphi_{n, \infty} (K^{(\infty)})\right)   \right). \]

(1) By using the facts that $K^{(n)}$ is an attracter  and  each $h^{(n)}_j$ is weakly contractive and the assumption (a),  
\[  d^{\textup{Haus}}_{M^{(n)}} \left( \varphi_{n, n} (K^{(n)}), \varphi_{n, n} \circ h^{(n)}_j  \circ \varphi_{n, n}^{-1} \left(L_j^{(n)}\right) \right) \le d^{\textup{Haus}}_{M^{(n)}} \left(\varphi_{n, n} (\widetilde K_{j}^{(n)}), L_j^{(n)}\right) \to 0. \]

(2) By the assumption (b),
\[ d^{\textup{Haus}}_{M^{(n)}} \left( \varphi_{n, n} \circ h^{(n)}_j  \circ \varphi_{n, n}^{-1} \left(L_{j}^{(n)}\right), \varphi_{n, \infty} \circ h^{(\infty)}_j  \circ \varphi_{n, \infty}^{-1} \left(L_{j}^{(n)}\right) \right)  \to 0. \]

(3) By using that facts that each $h^{(\infty)}_j$ is a weak contraction and \eqref{inf-2},  
\[ d^{\textup{Haus}}_{M^{(n)}} \left( \varphi_{n, \infty} \circ h^{(\infty)}_j  \circ \varphi_{n, \infty}^{-1} \left(L_{j}^{(n)}\right), \varphi_{n, \infty} \circ h^{(\infty)}_j  \circ \varphi_{n, \infty}^{-1} \left(\varphi_{n, \infty} (K^{(\infty)})\right) \right) \]
\[ \le  \phi_j \left(d^{\textup{Haus}}_{M^{(n)}} \left( L_{j}^{(n)},  \varphi_{n, \infty} (K^{(\infty)})  \right) \right). \]
\[ \le \phi_j \left(d^{\textup{Haus}}_{M^{(n)}} \left(\varphi_{n, n} (\widetilde K_{j}^{(n)}),  \varphi_{n, \infty} (K^{(\infty)})  \right) + d^{\textup{Haus}}_{M^{(n)}} \left( L_{j}^{(n)}, \varphi_{n, n} (\widetilde K_j^{(n)}) \right) \right). \]

Since for each $j$, 
\[ \lim_{n \to \infty} d^{\textup{Haus}}_{M^{(n)}} \left(\varphi_{n, n} (\widetilde K^{(n)}),  \varphi_{n, n} (\widetilde K_j^{(n)}) \right)= 0,\] 
by recalling that \[ d^{\textup{Haus}}_{M^{(n)}} \left(\varphi_{n, n} (\widetilde K^{(n)}),  \varphi_{n, \infty} (K^{(\infty)})  \right) > \epsilon_0, \]
it follows that for large $n$, 
\[ d^{\textup{Haus}}_{M^{(n)}} \left(\varphi_{n, n} (\widetilde K_{j}^{(n)}),  \varphi_{n, \infty} (K^{(\infty)})  \right) > \epsilon_0. \]

Hence, for large $n$, 
\[ \phi_j \left(d^{\textup{Haus}}_{M^{(n)}} \left(\varphi_{n, n} (\widetilde K_{j}^{(n)}),  \varphi_{n, \infty} (K^{(\infty)})  \right) + d^{\textup{Haus}}_{M^{(n)}} \left( L_{j}^{(n)}, \varphi_{n, n} (\widetilde K_j^{(n)}) \right) \right) \]
\[ \le d^{\textup{Haus}}_{M^{(n)}} \left(\varphi_{n, n} (\widetilde K_j^{(n)}),  \varphi_{n, \infty} (K^{(\infty)})  \right) + d^{\textup{Haus}}_{M^{(n)}} \left( L_{j}^{(n)}, \varphi_{n, n} (\widetilde K_{j}^{(n)}) \right)  - \epsilon_1 \]
\[ \le d^{\textup{Haus}}_{M^{(n)}} \left(\varphi_{n, n} (\widetilde K^{(n)}),  \varphi_{n, \infty} (K^{(\infty)})  \right) +  d^{\textup{Haus}}_{M^{(n)}} \left(\varphi_{n, n} (\widetilde K_j^{(n)}),  \varphi_{n, n} (\widetilde K^{(n)})  \right) \] 
\[ + d^{\textup{Haus}}_{M^{(n)}} \left( L_{j}^{(n)}, \varphi_{n, n} (\widetilde K_{j}^{(n)}) \right)  - \epsilon_1 \]
\[ \le  d^{\textup{Haus}}_{M^{(n)}} \left(\varphi_{n, n} (\widetilde K^{(n)}),  \varphi_{n, \infty} (K^{(\infty)})  \right) - \epsilon_1 / 2. \]

Hence if $n$ is sufficiently large,
\[ d^{\textup{Haus}}_{M^{(n)}} \left(\varphi_{n, n} (\widetilde K^{(n)}),  \varphi_{n, \infty} (K^{(\infty)})  \right) \] 
\[\le \max_j d^{\textup{Haus}}_{M^{(n)}} \left( \varphi_{n, n} \left( h^{(n)}_j (\widetilde K^{(n)}_j)\right),  \varphi_{n, \infty} (h^{(\infty)}_j (K^{(\infty)}) )  \right) \]
\[ \le  d^{\textup{Haus}}_{M^{(n)}} \left(\varphi_{n, n} (\widetilde K^{(n)}),  \varphi_{n, \infty} (K^{(\infty)})  \right) - \epsilon_1 / 4.  \]
This is a contradiction. 
\end{proof} 

\subsection{Application}

For $n \in \mathbb{N} \cup \{\infty\}$, 
let $(X^{(n)}, d_{X^{(n)}})$ and $(Y^{(n)}, d_{Y^{(n)}})$ be two compact metric spaces.  
Let $f^{(n)}_{j} : X^{(n)} \to X^{(n)}$ and $g^{(n)}_{j} : Y^{(n)} \to Y^{(n)}$ be weak contractions in the sense of Browder.  
Consider a conjugate equation \eqref{conjugate} on $\left(X^{(n)}, Y^{(n)}, (f^{(n)}_{i})_{i \in I}, (g^{(n)}_{i})_{i \in I}\right)$ satisfying Assumptions \ref{ass-f} - \ref{ass-new-compat}.

Let 
\[ h^{(n)}_i (x, y) := \left( f^{(n)}_i (x), g^{(n)}_{i}(y) \right). \]

Let $Z^{(n)} := X^{(n)} \times Y^{(n)}$ and 
\[ d_{Z^{(n)}}\left((x^{(n)}_1, y^{(n)}_1), (x^{(n)}_2, y^{(n)}_2)\right) := \sqrt{ d_{X^{(n)}}\left(x^{(n)}_1, x^{(n)}_2\right)^2 +  d_{Y^{(n)}}\left(y^{(n)}_1, y^{(n)}_2\right)^2 }. \]
This gives a metric on $Z$ and by this metric $Z$ is a compact metric space.

The following is easy to see. 
\begin{Prop}
Let $f_n, n \ge 1$, and $f$ be a uniformly continuous family of real functions on a common compact metric space. 
Then, $f_n \to f, n \to \infty$ uniformly if and only if the graphs of $f_n$ converges to the graph of $f$ with respect to the Hausdorff distance.    
\end{Prop}

\begin{Rem}
If the value of $g^{(n)}_{i}(x, y)$ depends on $x$, then, 
$h^{(n)}_i$ may not be a weak contraction on $X \times Y$ and the proof of Theorem \ref{GH-conv} does not applicable to that case.  
\end{Rem}

\subsection{Examples}  

Now we give three cases. We consider the case that $Z^{(\infty)} = M^{(n)}$ only.  

\begin{Exa}[Case 1,   $Z^{(n)} = Z^{(\infty)} = M^{(n)}$ and both of $\varphi_{n,n}$ and $\varphi_{n,\infty}$ are the identity map.] 
This is the case that the spaces are common and functions driving \eqref{conjugate} vary. 
\cite[Proposition 2.7]{O16} states a result of this kind, and Theorem \ref{GH-conv} will imply \cite[Proposition 2.7]{O16}. 
\end{Exa} 

\begin{Exa}[Case 2,  $Z^{(n)} \subset Z^{(\infty)} = M^{(n)}$ and both of $\varphi_{n,n}$ and $\varphi_{n,\infty}$ are the inclusion maps.]   
This gives a {\it discrete approximation} of solution. 
Since each $h^{(n)}_{i}$ is weakly contractive, it follows that 
\[ K^{(\infty)} = \overline{\bigcup_{k \ge 1} \bigcup_{i_1, \dots, i_k \in I} h^{(\infty)}_{i_1} \circ \cdots \circ h^{(\infty)}_{i_{k-1}} \left( \Fix( h^{(\infty)}_{i_k} ) \right)}.  \]
See \cite{Ha85}.  
Therefore if we let 
\[ K^{(n)} :=  \bigcup_{i_1, \dots, i_n \in I} h^{(\infty)}_{i_1} \circ \cdots \circ h^{(\infty)}_{i_{n-1}} \left( \Fix( h^{(\infty)}_{i_n} ) \right), \]
and, 
\[ K^{(n)}_{i} := K^{(n-1)}, i \in I, \] 
then, \eqref{approx} and \eqref{approx-conv} hold. 
\end{Exa}

\begin{Exa}[Case 3,  $Z^{(n)} \ne Z^{(\infty)} = M^{(n)}$ and $\varphi_{n,\infty}$ are identity maps.] 
We give an example for small deformations of de Rham type functions.   
Let $I := \{0,1\}$. 
For $n \in \mathbb{N}$,    
$X^{(n)} := [-1/n, (n+1)/n]$, $Y^{(n)} := [1/n, (n-1)/n]$.  
$X^{(\infty)} = Y^{(\infty)} := [0, 1]$.   
Let $e_{X,n} : X^{(n)} \to [0,1]$ and $e_{Y,n} : Y^{(n)} \to [0,1]$ be affine maps such that 
\[ e_{X,n}(-1/n) = 0, e_{X,n}((n+1)/n) = 1, e_{Y,n}(1/n) = 0, e_{Y,n}((n-1)/n) = 1. \]

Let $f_0(x) = x/2$, $f_1(x) = (x+1)/2$, $g_0 (y) = y/3$ and $g_1 (y) = (2y + 1)/3$.  
Then by Proposition \ref{interval-exist-con},  
there exists a unique continuous increasing solution $\varphi^{(\infty)}$ of \eqref{conjugate} for $\left(X^{(\infty)}, \{f^{(\infty)}_i\}_{i \in I}, Y^{(\infty)}, \{g^{(\infty)}_i\}_{i \in I}\right)$. 
Let $f^{(n)}_i := e_{X,n}^{-1} \circ f_i \circ e_{X,n}$, and $g^{(n)}_i := e_{Y,n}^{-1} \circ g_i \circ e_{Y,n}$, $i \in I$.   
Then there exists a unique continuous increasing solution $\varphi^{(n)}$ of \eqref{conjugate} for $\left(X^{(n)}, \{f^{(n)}_i\}_{i \in I}, Y^{(n)}, \{g^{(n)}_i\}_{i \in I}\right)$.

Let $M^{(n)} := [0,1]^2$.  
Let $\varphi_{n, n}(x, y) := (e_{X,n}(x), e_{Y,n}(y))$ and $\varphi_{n, \infty}(x, y) := (x, y)$.  
Now by applying Theorem \ref{GH-conv},     
\[ d^{\textup{GH}}\left(\textup{Graph}(\varphi^{(n)}), \textup{Graph}(\varphi^{(\infty)})\right) \to 0, \ \ n \to \infty. \]
\end{Exa} 

\section{Open problems}

(1) As in Example \ref{exa-overlap}, 
let $I = \{0,1\}$ and $X = Y = [0,1]$. 
Assume that $f_0 (x) = ax$, $f_1 (x) = ax + 1-a$, $g_0 (x) = bx$ and $g_1 (x) = bx + 1-b$ for some $a, b \in [1/2, 1)$.    
Then, find a pair $(a,b)$ such that the solution of \eqref{conjugate} exists. 
The case that $a$ is not an algebraic number seems interesting.  

(2) Several notions of bounded variation functions on metric measure spaces are proposed by Miranda \cite{Mir03}, Ambrosio-Di Marino \cite{AD14} etc.  
It is interesting to consider whether the solution of \eqref{conjugate} is of bounded variation in the senses of their papers. 
Furthermore, in the case that the solution of \eqref{conjugate} is {\it not} of bounded variation, it is also interesting to ask whether there exist a metric and a measure on $X$ such that the solution of \eqref{conjugate} is of bounded variation. 

(3) Muramoto and Sekiguchi \cite{MS} considered a generalization of de Rham type equations on $[0,1]$ which is different from ours. 
They introduced a new class of Takagi function by using it. 

(4) By following Hata-Yamaguti \cite{HY84}, derivative of the solution of de Rham function with respect to a parameter in $\{g_i\}_i$ yields a fractal function. 
From this viewpoint, it seems to be able to define an analogue of Takagi function on Sierpi\'nski gasket. 
Lebesgue singular function and the Takagi function on 2-dimensional plane is considered by Sumi \cite{Su07} in terms of random dynamical system on the complex plane. \\

{\it Acknowledgements.} The author was supported by JSPS KAKENHI Grant-in-Aid for JSPS Fellows (16J04213) and for Research activity Start-up (18H05830).
This work was also supported by the Research Institute for Mathematical Sciences, a Joint Usage/Research Center located in Kyoto University.\\

\end{document}